\documentclass{amsart}
\usepackage{amsmath, amssymb, latexsym}
\title{On the multiplicity conjecture for non-Cohen-Macaulay simplicial complexes}
\author{Michael Goff}

\newtheorem{theorem}{Theorem}[section]
\newtheorem{proposition}[theorem]{Proposition}
\newtheorem{corollary}[theorem]{Corollary}
\newtheorem{lemma}[theorem]{Lemma}

\newtheorem{conjecture}[theorem]{Conjecture}

\newcommand{\K}{\Gamma}
\newcommand{\field}{{\bf k}}
\newcommand{\codim}{\mbox{\upshape codim}\,}

\newcommand{\Skel}{\mbox{\upshape Skel}\,}
\newcommand{\lk}{\mbox{\upshape lk}\,}
\newcommand{\st}{\mbox{\upshape st}\,}
\newcommand{\Hom}{\mbox{\upshape Hom}\,}
\newcommand{\Ext}{\mbox{\upshape Ext}\,}

\def\proof{\smallskip\noindent {\it Proof: \ }}
\def\proofof#1{\smallskip\noindent {\it Proof of #1: \ }}
\def\endproof{\hfill$\square$\medskip}

\begin{document}

\begin{abstract}
We prove a reformulation of the multiplicity upper bound conjecture and use that reformulation to prove it for three-dimensional simplicial complexes and homology manifolds with many vertices.  We provide necessary conditions for a Cohen-Macaulay complex with many vertices to have a pure minimal free resolution and a characterization of flag complexes whose minimal free resolution is pure.
\end{abstract}

\date{February 1, 2008}

\maketitle

\section{Introduction and preliminaries}
The Multiplicity Conjecture of Herzog, Huneke, and Srinivasan provides a powerful connection between combinatorics and commutative algebra.  Let us first review the conjecture.

Throughout the paper we consider the polynomial ring
   $S=\field[x_1, \ldots, x_n]$
over an arbitrary field $\field$. If $N$ is a finitely-generated graded $S$-module, 
then the ($\mathbb{Z}$-graded) \textit{Betti numbers} of $N$, 
$\beta_{i,j}=\beta_{i,j}(N)$,
   are the invariants that appear
in the minimal free resolution of $N$ as an $S$-module:
\[ 0 \rightarrow
\bigoplus_j S(-j)^{\beta_{l,j}} \rightarrow \ldots \rightarrow 
\bigoplus_j S(-j)^{\beta_{2,j}} \rightarrow \bigoplus_j 
S(-j)^{\beta_{1,j}} \rightarrow S \rightarrow S/I \rightarrow 0 \] 
In the above expression, $S(-j)$ denotes $S$ with grading shifted by $j$, and $l$ 
denotes the length of the resolution. In particular, $l \geq \codim(N)$.

Our main objects of study are the \textit{maximal and minimal shifts} 
in the resolution of $N$, defined by 
$M_i=M_i(N)=\max\{j : \beta_{i,j}\neq 0\}$ and 
$m_i=m_i(N)=\min \{j : \beta_{i,j}\neq 0\}$ for $i=1, \ldots, l$, respectively.
The following conjecture, due to Herzog, Huneke, and 
Srinivasan \cite{HerzSr98}, is known as the multiplicity conjecture.

\begin{conjecture}  \label{multiplicity-conj} 
Let $I\subset S$ be a homogeneous ideal of codimension $c$. 
Then the multiplicity of $N$, $e(N)$, satisfies the following upper bound:
$$ e(N) \leq ( \prod_{i=1}^{c} M_i )/c!.
$$
Moreover, if $N$ is Cohen-Macaulay, then also

$$ e(N) \geq ( \prod_{i=1}^{c} m_i )/c!.
$$
\end{conjecture}

For simplicity, we denote $(\prod_{i=1}^{c} M_i )/c!$ by $U(N)$ and $(\prod_{i=1}^{c} m_i )/c!$ by $L(N)$.

It is furthermore conjectured (\cite{HerzZheng, MiglNagRom2}) that if $e(N)$ attains the upper bound, or if $e(N)$ attains the lower bound and $N$ is Cohen-Macaulay, then $N$ has a \textit{pure resolution}, which means that $m_i = M_i$ for $1 \leq i \leq c$.

In their remarkable new paper, Eisenbud and Schreyer \cite{EisSch} prove the Boij-S\"{o}derberg conjecture \cite{BoijSo}, which in turn implies both bounds of the multiplicity conjecture in the case that $N$ is Cohen-Macaulay.

This conjecture was motivated by the following result due to 
Huneke and Miller \cite{HuMill}.  If $I$ is a homogeneous ideal of $S$ and if $S/I$ is 
Cohen-Macaulay and has a pure resolution, 
then $e(S/I)= (\prod_{i=1}^{c} m_i )/c!$. 
Starting with the  paper of Herzog and Srinivasan \cite{HerzSr98}, 
a tremendous amount of effort has been made in establishing 
Conjecture \ref{multiplicity-conj} for various classes of rings $S/I$.
In particular, the non-Cohen-Macaulay case of the conjecture was proved in the following cases: $I$ is a stable or squarefree strongly 
stable ideal \cite{HerzSr98}, 
$I$ is a codimension 2 ideal \cite{Gold03, HerzSr98, Rom05}, 
and $I$ is a codimension 3 Gorenstein ideal \cite{MiglNagRom}.

We investigate Conjecture \ref{multiplicity-conj} 
for squarefree monomial ideals or, equivalently, 
Stanley-Reisner ideals of simplicial complexes.  A \textit{simplicial complex} $\K$ is a collection of subsets, called \textit{faces}, of $[n]$, such that $\K$ is closed under inclusion and for all $i \in [n]$, $\{i\} \in \K$.  We will also refer to $[n]$ as $V(\K)$, or the \textit{vertex set} of $\K$.  If $F \in \K$, the \textit{dimension} of $F$ is $|F|-1$.  The \textit{dimension} of $\K$ is the largest dimension of its faces.

The link of a face $F$, denoted $\lk_\K(F)$, is defined by $\{G-F: F \subseteq G, G \in \K\}$.  If $W \subseteq V(\K)$, then $\K[W]$ denotes the \textit{induced subcomplex} on $W$.  $\K[W]$ has vertex set $W$, and $F \in \K[W]$ if $F \in \K$ and $F \subseteq W$.  If $v \in V(\K)$, we abbreviate $\K[V-\{v\}]$ by $\K-v$.  We say that $\K$ is $r$-\textit{neighborly} if every $r$ vertices of $\K$ form a face.  If $\K$ is a simplicial complex on the vertex set 
$[n]:=\{1, 2, \ldots, n\}$, then its \textit{Stanley-Reisner ideal} 
(or the \textit{face ideal}), $I_\K$, 
is the ideal generated by the squarefree monomials 
corresponding to non-faces of $\K$, that is, $$ I_\K = (
       x_{i_1}\cdots x_{i_k} \, : \, \{i_1<\cdots <i_k\}\notin \K
              ), $$
and the \textit{Stanley-Reisner ring} (or the \textit{face ring}) 
of $\K$ is $S/I_\K$ \cite{St96}.

We denote the reduced simplicial homology of $\K$ with coefficients in $\field$ by $\tilde{H}(\K; \field)$.  We use $\tilde{H}(\K)$ when $\field$ is implicit.  Also, $\tilde{\beta}_p(\K;\field) = \tilde{\beta}_p(\K) = \dim_\field \tilde{H}(\K; \field)$ denotes the $p$-th Betti number of $\K$.

Various combinatorial and topological invariants of $\K$ 
are encoded in the algebraic invariants of $I_\K$ 
and vice versa \cite{BrHerz, St96}.
The Krull dimension of $S/I_\K$, $\dim  S/I_\K$, 
and the topological dimension of $\K$, $\dim \K$, satisfy
$\dim  S/I_\K =\dim \K +1$ and so
 $$       
\codim(I_\K)=n-\dim\K-1.
$$
The Hilbert series of $S/I_\K$ is 
determined by knowing the number of faces in each dimension.  
Specifically, let  $f_i$  be the number of $i$-dimensional faces.
 By convention, the empty set is the unique face of dimension $-1$. 
 Then, 
$$
\sum^\infty_{i=0} \dim_\field (S/I_\K)_i \lambda^i = 
\frac{h_0 + h_1 \lambda +\dots + h_d \lambda^d}{(1-\lambda)^d},
$$ 
where $(S/I_\K)_i$ is the $i$-th graded component of 
$S/I_\K$, $d=\dim\K+1=\dim  S/I_\K$, and
 \begin{equation} \label{h-vector}
    h_i = \sum^i_{j=0} (-1)^{i-j} \binom{d-j}{d-i} f_{j-1}.
\end{equation}
The multiplicity $e(S/I_\K)$ is 
the number of $(d-1)$-dimensional faces of $\K$ which in
 turn is $h_0 + \dots + h_d.$ The minimal and maximal shifts 
have the following interpretation in terms of the reduced homology:
\begin{eqnarray}
\label{m-interpr} m_i(S/I_\K)&=&\min\{|W| \, : \, W\subseteq[n] \mbox{ and }
              \tilde{H}_{|W|-i-1}(\K[W]; \field)\neq 0\}. \\
\label{M-interpr} M_i(S/I_\K)&=&\max\{|W| \, : \, W\subseteq[n] \mbox{ and }
              \tilde{H}_{|W|-i-1}(\K[W]; \field)\neq 0\},
\end{eqnarray}
The above expressions for $m_i$ and $M_i$ follow easily from Hochster's 
formula on the Betti numbers $\beta_{i,j}(S/I_\K)$ 
\cite[Theorem II.4.8]{St96}. Thus, for face ideals,
 Conjecture \ref{multiplicity-conj} can be considered as 
a purely combinatorial-topological statement.  In the case of Stanley-Reisner rings, we refer to the multiplicity on simplicial complexes and write $L(\K) := L(S/I_\K)$ and $U(\K) := U(S/I_\K)$.

We say that a $(d-1)$-dimensional simplicial complex $\K$ with $n$ vertices is \textit{Cohen-Macaulay} if its Stanley-Reisner ring is Cohen-Macaulay.  Equivalently, by Reisner's criterion \cite{Hochs}, $\K$ is Cohen-Macaulay if the following condition holds: for all $F \in \K$, $\tilde{H}_p(\lk_\K(F); \field) = 0$ if $-1 \leq p < d-1-|F|$.  Equivalently, by Hochster's theorem \cite{Reisner}, for $W \subset V(\K)$, $\tilde{H}_p(\K[W]; \field) = 0$ when $|W| > n-d+1+p$.  Similarly, $\K$ is \textit{Gorenstein} if its Stanley-Reisner ring is Gorenstein.  

The \textit{simplicial join} of simplicial complexes $\K_1$ and $\K_2$, denoted by $\K_1 * \K_2$, is given by $V(\K_1 * \K_2) = V(\K_1) \coprod V(\K_2)$ and $F \in \K_1 * \K_2$ if $F \cap V(\K_1) \in \K_1$ and $F \cap V(\K_2) \in \K_2$.  Say that $\K$ is \textit{Gorenstein*} if $\K$ is Gorenstein and not a cone.  Every Gorenstein complex is the simplicial join of a Gorenstein* complex and a simplex.

In \cite{NovSw}, the multiplicity conjecture is established for matroid complexes, complexes of dimension at most $2$, and Gorenstein* complexes of dimension at most $4$.  Our work uses many of the same techniques as in \cite{NovSw}.  Our main results are the following.

\begin{itemize}
\item In Section \ref{3DSection}, we prove the upper bound part of Conjecture \ref{multiplicity-conj} 
for the Stanley-Reisner rings of three-dimensional simplicial complexes.
\item In Section \ref{HomManifold}, we prove the upper bound part of Conjecture \ref{multiplicity-conj} 
for the Stanley-Reisner rings of $(d-1)$-dimensional homology manifolds with sufficiently many vertices, where $d$ is odd (Theorem \ref{HomManOddD}) or $d \leq 10$ and the manifold is orientable over $\field$ (Theorem \ref{EvenD}).
\item In Theorem \ref{LLB}, we provide conditions under which a simplicial complex with sufficiently many vertices can attain the lower bound of Conjecture \ref{multiplicity-conj}.
\item As a corollary to Theorem \ref{LLB}, we characterize in Corollary \ref{GStar} and Lemma \ref{d2Neighborly} when the lower bound is attained for $(d-1)$-dimensional Gorenstein complexes with sufficiently many vertices.
\item In Theorem \ref{PureFlag}, we classify all quadratic monomial ideals that have a pure resolution.
\end{itemize}

By the minimality of the resolution, the minimal shifts of a graded $S$-module are strictly increasing: $m_i < m_{i+1}$.  If $\K$ is a $(d-1)$-dimensional Cohen-Macaulay simplicial complex with $n$ vertices, then the codimension of $I_\K$ is $n-d$.  Hence there are $d$ integers in the set $[n]$ that are not minimal shifts of $\K$; we call them \textit{lower skips} of $\K$.  By Lemma \ref{IncM}, there are also $d$ integers in $[n]$ that are not in the first $n-d$ maximal shifts of $\K$; we call them \textit{upper skips} of $\K$ and denote them $Q_i(\K)$, $1 = Q_0(\K) < Q_1(\K) < \ldots < Q_{d-1}(\K)$.

Recall that the \textit{i-skeleton} of a simplicial complex $\K$, $\Skel_i \K$, is the complex consisting of all faces of $\K$ of dimension $i$ or less.  If $\K$ has $n$ vertices and dimension $d-1$, we say that $\K$ is $i$-\textit{Cohen-Macaulay} if $\K[W]$ is Cohen-Macaulay of dimension $d-1$ whenever $|W| > n-i$.  For $0 \leq i \leq d-1$, let $q_i(\K)$ be the maximum $j$ such that $\Skel_i(\K)$ is $j$-Cohen-Macaulay.  The \textit{Cohen-Macaulay connectivity sequence} of $\K$ is $(q_0(\K),q_1(\K),\ldots,q_{d-1}(\K))$.  Since all $0$-dimensional complexes are Cohen-Macaulay, $q_0(\K)=n$, while $q_1(\K)$ is the maximum $j$ such that the graph of $\K$ is $j$-connected.  It is proven in \cite{Floyst} that $q_{d-1}(\K) < q_{d-2}(\K) < \ldots < q_1(\K) < q_0(\K)$, and it is proven in \cite{NovSw} that the set $\{n-q_i(\K)+1: 0 \leq i \leq d-1\}$ is precisely the set of lower skips of $\K$.  We will use the reformulation of the multiplicity lower bound conjecture from \cite{NovSw}:

\begin{theorem}
\label{CMLB}
Let $\K$ be a $(d-1)$-dimensional simplicial complex with $n$ vertices and Cohen-Macaulay sequence $(q_0,\ldots,q_{d-1})$.  Then $\K$ satisfies the multiplicity lower bound conjecture if and only if
$$
f_{d-1}(\K) \geq \frac{n(n-1)\ldots(n-d+1)}{(n-q_0+1)(n-q_1+1)\ldots (n-q_{d-1}+1)}.
$$
\end{theorem}

Suppose $\K$ is a $(d-1)$-dimensional simplicial complex with $n$ vertices.  We say that $\K$ is \textit{t-Leray} if for all $W \subset V(\K)$ and $p \geq t$, $\tilde{H}_p(\K[W]) = 0$.  By Equation (\ref{M-interpr}), $\K$ is $t$-Leray if and only if $M_{i}(\K) \leq i+t$ for all $1 \leq i \leq n-d$.  By Hochster's formula, $\K$ is $t$-Leray if and only if $r(S/I_\K) \leq t$, where $r(S/I_\K)$ is the Castelnuovo-Mumford regularity of $S/I_\K$.

For a $(d-1)$-dimensional simplicial complex $\K$ with $n$ vertices, and for $0 \leq p < d$ and $1 \leq m \leq n$, we define $A(p,m)(\K)$, or $A(p,m)$ when $\K$ is implicit, to be $$A(p,m)(\K) := \frac{\sum_{|W|=m}\tilde{\beta}_p(\K[W])}{{n \choose m}}.$$ $A(p,m)(K)$ is the average value of the $p$-th Betti number over all induced subcomplexes of $\K$ of size $m$.  By Hochster's formula, $\beta_{m-p-1,m}(\K) = {n \choose m}A(p,m)(\K)$.

One of our main tools for verifying the multiplicity upper bound conjecture is the following.  For appropriate values $c_i$ that depend on the $Q_i(\K)$, $\K$ satisfies the multiplicity upper bound conjecture if and only if $$\sum_{i=0}^{d-1}c_i \sum_{j=0}^{i-1}(-1)^{i-j-1}A(j,Q_i) \geq 0.$$ There is also an analogous restatement of the multiplicity lower bound conjecture.  We discuss both of these restatements in the next section.

\section{A reformulation of the multiplicity conjecture}
\label{SkipSection}
In this section we prove useful reformulations of the multiplicity upper bound and lower bound conjectures.  By the minimality of the resolution, the minimal shift sequence is strictly increasing.  First we verify a similar statement regarding the maximal shift sequence of a graded module.

\begin{lemma}
\label{IncM}
Let $N$ be a graded $S$-module with codimension $c$.  Then for $1 \leq i \leq c-1$, $M_{i+1}(N) > M_i(N)$.
\end{lemma}
\proof 
Let $F$ be the minimal free resolution of $N$, and consider $F^* := \Hom_S(F,S)$.  It is well known that for $i < c$, $\Ext^i_S(N,S) = 0$ (e.g. follows from I.12.3 of \cite{St96}), which implies $H^i(F^*) = H_i(F) = 0$ for $i < c$.  Hence the portion $i < c$ of $F^*$ is exact and a minimal free resolution.  The degrees of the generators of $F^*_j$ are the negatives of the degrees of the generators of $F_j$.  Hence the top degree generators of $F_j$ correspond to the bottom degree generators of $F^*_j$ for $j < c$.  Since the minimal shift sequence of $F^*$ is strictly increasing, $M_i(F) < M_{i+1}(F)$ for $i < c$.
\endproof

The maximal shift sequence might not be strictly increasing beyond the $c$th entry.  For example, let $\K$ be the simplicial complex given by the disjoint union of a cycle of length $3$ and a vertex.  Then $M(S/I_\K) = (3,4,4)$.

Now we want to use Lemma \ref{IncM} to find a recharacterization of the multiplicity conjecture.  The computations
below have a flavor similar to and are motivated by the computations from \cite{HerzSr98} used there to prove the multiplicity conjecture for ideals with a quasi-pure resolution.  Let $\K$ be a $(d-1)$-dimensional simplicial complex with $n$ vertices and upper skips $Q_i$.  Recall the Euler-Poincar\'{e} formula, which states $$f_{d-1}(\K)-f_{d-2}(\K)+\ldots+(-1)^d f_{-1}(\K) = \tilde{\beta}_{d-1}(\K;\field) - \ldots +(-1)^{d-1}\tilde{\beta}_0(\K;\field).$$ Choose $W \subset V(\K)$ such that $|W| = Q_j$ for some $1 \leq j \leq d-1$.  Then the Euler-Poincar\'{e} formula yields $$\sum_{i=-1}^{d-1}(-1)^{d-i-1}f_{i}(\K[W]) = \sum_{i=0}^{d-1}(-1)^{d-i-1}\tilde{\beta}_i(\K';\field).$$ By definition of $Q_j$, the right-hand side of this equation reduces to $\sum_{i=0}^{j-1}(-1)^{d-i-1}\tilde{\beta}_i(\K';\field)$.  Average over all such $W$ to conclude
\begin{equation}
\label{EP-Induced}
\sum_{i=-1}^{d-1}(-1)^{d-i-1}\frac{{Q_j \choose {i+1}}}{{n \choose {i+1}}}f_{i} = \sum_{i=0}^{j-1}(-1)^{d-i-1}A(i,Q_j).
\end{equation}
In Equation (\ref{EP-Induced}), take $f_0 = n$ and $f_{-1} = 1$ to build a system of linear equations in the remaining $f_i$:
\begin{equation}
\sum_{i=1}^{d-1}(-1)^{d-i-1}\frac{{Q_j \choose {i+1}}}{{n \choose {i+1}}}f_{i} = (-1)^{d}(Q_j-1) + \sum_{i=0}^{j-1}(-1)^{d-i-1}A(i,Q_j).
\end{equation}
Solving for $f_{d-1}$ using Cramer's rule yields $f_{d-1} = E/D$, where

\begin{equation}
\label{Cramer} E =
\begin{vmatrix} 
(-1)^{d-2}\frac{{Q_1 \choose 2}}{{n \choose 2}} & \hdots & (-1)^{d}(Q_1-1) + \sum_{i=0}^{1-1}(-1)^{d-i-1}A(i,Q_1) \\
\vdots & \ddots & \vdots \\
(-1)^{d-2}\frac{{Q_{d-1} \choose 2}}{{n \choose 2}} & \hdots & (-1)^{d}(Q_{d-1}-1) + \sum_{i=0}^{d-2}(-1)^{d-i-1}A(i,Q_{d-1}) \\
\end{vmatrix}
\quad \mbox{and}
\end{equation}
$$ D =
\begin{vmatrix} 
(-1)^{d-2}\frac{{Q_1 \choose 2}}{{n \choose 2}} & \hdots & (-1)^{0}\frac{{Q_1 \choose d}}{{n \choose d}} \\
\vdots & \ddots & \vdots \\
(-1)^{d-2}\frac{{Q_{d-1} \choose 2}}{{n \choose 2}} & \hdots & (-1)^{0}\frac{{Q_{d-1} \choose d}}{{n \choose d}} \\
\end{vmatrix}.
$$

For the moment, take $A(i,Q_j) = 0$.  The above reduces to

$$
f_{d-1} = (-1)^d{n \choose d}
\frac
{\begin{vmatrix} 
{Q_1 \choose 2} & \hdots & (Q_1-1)\\
\vdots & \ddots & \vdots \\
{Q_{d-1} \choose 2} & \hdots & (Q_{d-1}-1)\\
\end{vmatrix}}
{\begin{vmatrix} 
{Q_1 \choose 2} & \hdots & {Q_1 \choose d} \\
\vdots & \ddots & \vdots \\
{Q_{d-1} \choose 2} & \hdots & {Q_{d-1} \choose d} \\
\end{vmatrix}}
$$

\begin{equation}
\label{VDM}
= (-1)^d \frac{n(n-1)\ldots(n-d+1)}{Q_1\ldots Q_{d-1}}
\frac
{\begin{vmatrix} 
Q_1 & \hdots & 1\\
\vdots & \ddots & \vdots \\
Q_{d-1} & \hdots & 1\\
\end{vmatrix}}
{\begin{vmatrix} 
1 & \hdots & (Q_1-2)\ldots(Q_1-d+1) \\
\vdots & \ddots & \vdots \\
1 & \hdots & (Q_{d-1}-2)\ldots(Q_{d-1}-d+1) \\
\end{vmatrix}}.
\end{equation}

The previous equality follows by factoring terms out of each row in the matrices.  In the fraction of Equation (\ref{VDM}), the numerator is a polynomial of degree $(d-2)(d-1)/2$ in the variables $Q_i$, which vanishes when $Q_i = Q_j$ for all $1 \leq i < j \leq d-1$.  The monomial $Q_1Q_2^2 \ldots Q_{d-2}^{d-2}$ appears with coefficient $1$.  Hence the numerator is $(-1)^d\prod_{1 \leq i < j \leq d-1}(Q_j-Q_i)$.  All the same applies to the denominator, except that the monomial $Q_1Q_2^2 \ldots Q_{d-2}^{d-2}$ appears with coefficient $(-1)^d$; hence the denominator is $\prod_{1 \leq i < j \leq d-1}(Q_j-Q_i)$.  We conclude that $f_{d-1} = \frac{n\ldots(n-d+1)}{Q_1\ldots Q_{d-1}}$, precisely the conjectured upper bound.

For some positive numbers $c'_i$ that depend on the $Q_j$ and not on the $A(i,Q_j)$, $$f_{d-1} = \frac{n(n-1)\ldots(n-d+1)}{Q_1\ldots Q_{d-1}} + \sum_{i=0}^{d-1}c'_i \sum_{j=0}^{i-1}(-1)^{i-j}A(j,Q_i).$$  Let $\tilde{c}_i = Dc'_i$, where $D$ is the absolute value of the denominator of Equation (\ref{Cramer}).  Next we aim to calculate $\tilde{c}_i$ for $1 \leq j \leq d-1$.  By expanding along the rightmost column of the numerator of Equation (\ref{Cramer}), we find
\begin{equation}
\tilde{c}_i = (-1)^{d-1}(-1)^d
\begin{vmatrix}
\frac{{Q_1 \choose 2}}{{n \choose 2}} & \hdots & \frac{{Q_1 \choose d-1}}{{n \choose d-1}} \\
\vdots & \ddots & \vdots \\
\widehat{\frac{{Q_i \choose 2}}{{n \choose 2}}} & \hdots & \widehat{\frac{{Q_i \choose d-1}}{{n \choose d-1}}} \\
\vdots & \ddots & \vdots \\
\frac{{Q_{d-1} \choose 2}}{{n \choose 2}} & \hdots & \frac{{Q_{d-1} \choose d-1}}{{n \choose d-1}}
\end{vmatrix}.
\end{equation}
In the above expression, the hats indicate that the specified row is removed.  The factor of $(-1)^{d-1}$ is present since we are expanding along the $d$-th column, while the factor of $(-1)^d$ accounts for the removal of a factor $(-1)^d$ in Equation (\ref{Cramer}).

Let $c_i = -\tilde{c}_i{n \choose 2}{n \choose 3} \ldots {n \choose d-1}2!3!\ldots (d-1)!c_i$.  Then $$c_i = 
\begin{vmatrix}
(Q_1)(Q_1-1) & \hdots & (Q_1)(Q_1-1)\ldots(Q_1-d+1) \\
\vdots & \ddots & \vdots \\
\widehat{(Q_i)(Q_i-1)} & \hdots & \widehat{(Q_i)(Q_i-1)\ldots(Q_i-d+1)} \\
\vdots & \ddots & \vdots \\
(Q_{d-1})(Q_{d-1}-1) & \hdots & (Q_{d-1})(Q_{d-1}-1)\ldots(Q_{d-1}-d+1) \\
\end{vmatrix}
$$ $$= \prod_{1 \leq u \leq d-1, u \neq i}(Q_u(Q_u-1))\prod_{1 \leq u < v \leq d-1, u,v \neq i}(Q_v-Q_u).$$

The next theorem summarizes the foregoing discussion.

\begin{theorem}
\label{UBIEqiv}
$\K$ satisfies the multiplicity upper bound conjecture if and only if 
\begin{equation}
\label{UBIEqivEquation}
\sum_{i=0}^{d-1}c_i \sum_{j=0}^{i-1}(-1)^{i-j-1}A(j,Q_i) \geq 0.
\end{equation}
\end{theorem}
Note that each $c_i > 0$.  For our applications, it is often more useful to work with ${c_j}/{c_i}$.  When $j > i$, $$\frac{c_j}{c_i} = \frac{Q_i(Q_i-1)}{Q_j(Q_j-1)}\prod_{u=j+1}^{d-1}\frac{Q_u-Q_i}{Q_u-Q_j}\prod_{u=i+1}^{j-1}\frac{Q_u-Q_i}{Q_j-Q_u}\prod_{u=1}^{i-1}\frac{Q_i-Q_u}{Q_j-Q_u}.$$

A stronger version of the multiplicity conjecture also asserts that if $\K$ attains the upper bound, then $\K$ is Cohen-Macaulay and has a pure resolution.  Our next theorem proves that fact in one important case.

\begin{proposition}
\label{UPureRes}
Let $\K$ be a $(d-1)$-dimensional simplicial complex.  Suppose that, for all $0 \leq j < i < d$, $A(j,Q_i) = 0$.  Then $\K$ is Cohen-Macaulay and has a pure resolution.
\end{proposition}
\proof For $1 \leq i \leq d-1$, an induced subcomplex of $\K$ on a vertex set of size $Q_i$ is acyclic, and so the $Q_i$ must also be skips in the minimal shift sequence.  Hence for $0 \leq i \leq n-d$, $m_i(\K) = M_i(\K)$.  Also since $Q_0,Q_1,\ldots,Q_{d-1}$ are skips in the minimal shift sequence of $\K$ and the minimal shift sequence is strictly increasing even beyond the $(n-d)$-th term, the minimal free resolution has length $n-d$.  This implies that $\K$ is Cohen-Macaulay.
\endproof

Suppose that $A(j,Q_i)(\K) \neq 0$ for some $0 \leq j < i \leq d-1$.  Then either $Q_i-j-1 > n-d$ and $\K$ is not Cohen-Macaulay by Hochster's criterion, or $m_{Q_i-j-1}(\K) \leq Q_j$ and hence $X_i(\K) > Q_i(\K)$ and $\K$ does not have a pure resolution.  Therefore, the conjecture that if $\K$ satisfies the multiplicity upper bound with equality, then $\K$ is Cohen-Macaulay and has a pure resolution is equivalent to the following conjecture.

\begin{conjecture}
\label{UBIEqual}
Let $\K$ be a $(d-1)$-dimensional simplicial complex.  Then $$\sum_{i=0}^{d-1}c_i \sum_{j=0}^{i-1}(-1)^{i-j-1}A(j,Q_i) = 0$$ if and only if $A(j,Q_i) = 0$ for all $1 \leq j < i \leq d-1$.
\end{conjecture}

If $\K$ has a quasi-pure resolution (that is, $m_{i+1} \geq M_i$ for all $i$), then $A(j,Q_i) = 0$ unless $j = i-1$, and also $A(j,X_i) = 0$ unless $j = i$.  Hence from Theorem \ref{UBIEqiv}, we recover the multiplicity conjecture for Cohen-Macaulay simplicial complexes with quasi-pure resolutions from \cite{HerzSr98}, and equality is attained only when $\K$ has a pure resolution: $f_{d-1}(\K) = \prod_{i=1}^{n-d} m_i(\K) / (n-d)! = \prod_{i=1}^{n-d} M_i(\K) / (n-d)!$ (proven in \cite{HerzZheng}).

Our first application of Theorem \ref{UBIEqiv} is the following result that in particular implies the multiplicity upper bound conjecture for two-dimensional complexes and non-flag, three-dimensional complexes.  (Recall that a simplicial complex $\K$ is \textit{flag} if $I_\K$ is a quadratic ideal.)

\begin{theorem}
\label{LowSkips}
Let $\K$ be a $(d-1)$-dimensional complex, $d \geq 3$, with $n$ vertices and maximal shift sequence $M$ with skips $Q_i$.  If  $M_1 \geq d-1$, or equivalently if $Q_{d-3} = d-2$, then $\K$ satisfies the multiplicity upper bound conjecture.
\end{theorem}

First we need a lemma.  This calculation was used in the proof of the multiplicity upper bound conjecture for matroid complexes, and it is useful in a wider context.

\begin{lemma}
\label{HomRed}
Let $\K$ be a $(d-1)$-dimensional simplicial complex with $n$ vertices.  Suppose $\tilde{H}_{d-1}(\K) \neq 0$.  If $\K-v$ satisfies the multiplicity upper bound conjecture for all $v \in V(\K)$, then $\K$ satisfies the multiplicity upper bound conjecture.  Furthermore, suppose that for all $v \in V(\K)$, $\K-v$ attains the upper bound if and only if $\K-v$ is Cohen-Macaulay and has a pure resolution.  Then $\K$ attains the upper bound if and only if $\K$ is Cohen-Macaulay and has a pure resolution.
\end{lemma}
\begin{proof}
Since $\tilde{H}_{d-1}(\K) \neq 0$, $M_{n-d}(\K) = n$.  Since every $(d-1)$-dimensional face of $\K$ contains $d$ vertices,
\begin{equation}
\label{HomRedEquation}
f_{d-1}(\K) = \frac{1}{n-d}\sum_{v \in V(\K)}f_{d-1}(\K-v)
\end{equation}
$$
\leq \frac{1}{n-d}\sum_{v \in V(\K)}\frac{\prod_{i=1}^{n-d-1}M_i(\K-v)}{(n-d-1)!} \leq \frac{\prod_{i=1}^{n-d} M(\K)}{(n-d)!}.
$$

Now suppose $\K$ attains the upper bound.  Then each inequality above holds with equality, and so $\K-v$ also has dimension $d-1$ and attains the upper bound for all $v \in V(\K)$.  Furthermore in this case, $M_i(\K-v) = M_i(\K)$ for all $v \in V(\K)$ and $1 \leq i \leq n-d-1$.  By hypothesis, all $\K-v$ are Cohen-Macaulay and have pure resolutions.  It follows that $A(i,Q_j)(\K-v) = A(i,Q_j)(\K) = 0$ for $1 \leq i < j \leq d-1$ and $v \in V(\K)$.  Then $\K$ is Cohen-Macaulay and has a pure resolution by Proposition \ref{UPureRes}.
\end{proof}

\proofof{Theorem \ref{LowSkips}} Suppose $M_1 \geq d-1$.  First we will show that we may assume without loss of generality that $\K$ is $(d-2)$-neighborly.  Suppose $\K$ is not $(d-2)$-neighborly, and that $F$ is a minimal non-face of $\K$ with at most $d-2$  vertices.  Let $\K' := \K \cup \{F\}$.  Adding $F$ can only increase homology of induced subcomplexes in dimension $|F|-1$, decrease it in dimension $|F|-2$, and does not change homology in other dimensions.  Thus if $W \subset V(\K) = V(\K')$, then $\tilde{H}_p(\K'[W]) = \tilde{H}_p(\K[W])$ if $p \geq d-2$.  Since $M_1(\K) \geq d-1$, $M_i(\K) \geq d-2+i$ for $1 \leq i \leq n-d$, and for a fixed $i$ there exists $W \subset V(\K)$ such that $|W| = M_i$ and $\tilde{H}_{M_i-i-1}(\K[W]) \neq 0$.  Since $M_i-i-1 \geq d-3$, $\tilde{H}_{M_i-i-1}(\K'[W]) \geq \tilde{H}_{M_i-i-1}(\K[W]) \neq 0$, and it follows that $M_i(\K') \geq M_i(\K)$.  If $M_i(\K') > M_i(\K)$, then for some $W$ with $|W| = M_i(\K')$, $\tilde{H}_{M_i(\K')-i-1}(\K'[W]) \neq 0$ while $\tilde{H}_{M_i(\K')-i-1}(\K[W]) = 0$.  However, this is a contradiction since in this case, $M_i(\K')-i-1 = d-3$ and $M_{1}(\K) \geq d-1$.  We conclude that $M_i(\K) = M_i(\K')$.  Also, $f_{d-1}(\K) = f_{d-1}(\K')$.  Hence, we may assume that $\K$ is $(d-2)$-neighborly.

Next we will assume, without loss of generality, that $\tilde{H}_{d-1}(\K) = 0$.  Suppose $\tilde{H}_{d-1}(\K) \neq 0$, and let $v \in V(\K)$.  Then $\K-v$ is a $(d-2)$-neighborly induced subcomplex of $\K$, and $M_1(\K-v) \geq d-1$.  By induction on $n$, $\K-v$ satisfies the multiplicity upper bound conjecture.  Hence $\K$ satisfies the multiplicity upper bound conjecture by Lemma \ref{HomRed}.

Next we want to assume, without loss of generality, that $\tilde{H}_{d-3}(\K;k) = 0$.  Since $\K$ is $(d-2)$-neighborly, $\tilde{H}_{d-3}(\K;k)$ is generated (not necessarily minimally) by $\{\partial F: F \not\in \K, |F| = d-1$\}.  Suppose that for some $F \subset V(\K), |F| = d-1$ and $[\partial F] \neq 0$ in $\tilde{H}_{d-3}(\K;k)$.  Then $F \not \in \K$.  Let $\K' := \K \cup \{F\}$.  We will show that for each $1 \leq i \leq d-1$, $M_i(\K') = M_i(\K)$.  By repeating this procedure we may assume without loss of generality $\tilde{H}_{d-3}(\K;k) = 0$.

First we show that $Q_i(\K) \geq Q_i(\K')$ for all $1 \leq i \leq d-1$.  Since $\K'$ is $(d-2)$-neighborly, $Q_i(\K') = Q_i(\K) = i+1$, for $i \leq d-3$.  As $|F| = d-1$, we have $\tilde{H}_{d-2}(\K[W]) \leq \tilde{H}_{d-2}(\K'[W])$ and  $\tilde{H}_{d-1}(\K[W]) = \tilde{H}_{d-1}(\K'[W])$ for all $W \subset V(\K)$.  It follows that $Q_{d-2}(\K) \geq Q_{d-2}(\K')$ and $Q_{d-1}(\K) = Q_{d-1}(\K')$.

It only remains to be shown $Q_{d-2}(\K') = Q_{d-2}(\K)$, which follows if, for any $W \subset V(\K)$, $\tilde{H}_{d-2}(\K'[W]) = \tilde{H}_{d-2}(\K[W])$.  Assume $F \subset W$.  Since $[\partial F] \neq 0$ in $\tilde{H}_{d-3}(\K)$, $[\partial F] \neq 0$ in $\tilde{H}_{d-3}(\K[W])$, while $[\partial F] = 0$ in $\tilde{H}_{d-3}(\K'[W])$ by $F \in \K'[W]$.  The Euler-Poincar\'{e} formula together with $\tilde{H}_{d-3}(\K'[W]) \neq \tilde{H}_{d-3}(\K[W])$ yields $\tilde{H}_{d-2}(\K'[W]) = \tilde{H}_{d-2}(\K[W])$.

Since $\K$ is $(d-2)$-neighborly, $A(p,m) = 0$ for $p \leq d-4$ and $1 \leq m \leq n$.  Also by $\tilde{H}_{d-1}(\K;\field) = 0$, $Q_{d-1} = n$ and $A(d-3,Q_{d-1}) = 0$.  In Equation (\ref{UBIEqivEquation}), all $A(i,Q_j)$ terms with negative coefficients vanish.  Hence $\K$ satisfies the multiplicity upper bound conjecture by Theorem \ref{UBIEqiv}. \endproof

The last step of the proof also follows from Theorem 1.5 of \cite{HerzSr98}.

\begin{corollary}
Let $\K$ be as in Theorem \ref{LowSkips}.  If $\K$ attains the upper bound with equality, then $\K$ is Cohen-Macaulay and has a pure resolution.
\end{corollary}
\proof By Lemma \ref{HomRed}, we may assume $\tilde{H}_{d-1}(\K)=0$.  First construct $\K_1$ from $\K$ by adding all faces of dimension $d-3$ or less.  If $\tilde{H}_{d-3}(\K_1) \neq 0$, add a face $F$ to $\K_1$ such that $[F] \neq 0$ in $\tilde{H}_{d-3}(\K_1)$.  Construct $\K'$ from $\K_1$ by adding such faces until $\tilde{H}_{d-3}(\K') = 0$.

From the proof of Theorem \ref{LowSkips}, $\K'$ attains the upper bound if and only if $\K$ attains the upper bound.  In Equation (\ref{UBIEqivEquation}), all $A(i,Q_j)(\K')$ terms with negative coefficients vanish, so $\K'$ attains the upper bound if and only if $A(i,Q_j)(\K') = 0$ for $1 \leq i < j \leq d-1$.  By Proposition \ref{UPureRes}, $\K'$ is Cohen-Macaulay and has a pure resolution.

If $\K \neq \K'$, then $\K'$ is not pure.  Hence $\K'$ does not attain the upper bound, and $\K$ does not attain the upper bound. \endproof

\section{Upper Bound on Three-dimensional complexes}
\label{3DSection}
For this section, $\K$ always refers to a three-dimensional simplicial complex, and $\K$ has upper skips $Q_1,Q_2,Q_3$.  Our main theorem is the following.
\begin{theorem}
\label{DimThree}
A three-dimensional simplicial complex $\K$ satisfies the multiplicity upper bound conjecture.  Furthermore, if $\K$ attains the multiplicity upper bound, then $\K$ is Cohen-Macaulay and has a pure resolution.
\end{theorem}
The outline of the proof is as follows.  We show
\begin{equation}
\label{Main3D}
c_3(A(0,Q_3)-A(1,Q_3)-A(2,Q_3)) + c_2(A(1,Q_2)-A(0,Q_2)) + c_1A(0,Q_1) \geq 0,
\end{equation}
and that if (\ref{Main3D}) is an equality, then each $A(i,Q_j)=0$.  This implies Theorem \ref{DimThree} by Theorem \ref{UBIEqiv}.  By Lemma \ref{HomRed}, we may assume $\K$ is $3$-Leray, or that $Q_3 = f_0(\K) := n$.  More generally, we will show that for integers $R_1 \geq Q_1, Q_2 \leq R_2 < n$, and $R_3=Q_3$, the following holds:
\begin{equation}
\label{Main3DR}
c_3(A(0,R_3)-A(1,R_3)) + c_2(A(1,R_2)-A(0,R_2)) + c_1A(0,R_1) \geq 0.
\end{equation}
In (\ref{Main3DR}) and in all the following lemmas that use the $R_i$, take
$$c_i = \prod_{1 \leq u \leq 3, u \neq i}(R_u(R_u-1))\prod_{1 \leq u < v \leq 3, u,v \neq i}(R_v-R_u).$$
We prove the above result by induction on $R_3 - R_2$, and we may then assume $R_2 = R_3-1=n-1$.  Then we will show $c_3A(0,R_3) + c_2A(0,R_2) + c_1A(0,R_1) \geq 0$, with equality only when $A(0,R_3) = A(0,R_2) = A(0,R_1) = 0$.  Next we argue that we may assume, without loss of generality, that no edge of $\K$ is a maximal face.  Then we will show that one of the following conditions holds: $A(1,R_3)=0$, $c_2A(0,R_2) > c_3(A(1,R_3)$, or $\K$ contains two vertices $p$ and $q$ such that $||\K||-\{p,q\}$ is disconnected (here $||\K||$ denotes the \textit{geometric realization} of $\K$).  Finally, in the latter case, we use a shifting operation on $\K$ that preserves $c_3(A(0,R_3)-A(1,R_3)) + c_2(A(1,R_2)-A(0,R_2)) + c_1A(0,R_1)$ and strictly reduces $A(1,R_3)$.  This proves Theorem \ref{DimThree}.

Our main effort will be to prove the following lemma.

\begin{lemma}
\label{CloseQ3Q2}
Let $R_1 \geq Q_1, R_2 = Q_3-1$, and $R_3 = Q_3 = n$.  Then $c_3(A(0,R_3)-A(1,R_3)) + c_2(A(1,R_2)-A(0,R_2)) + c_1A(0,R_1) \geq 0$.  Also, equality is attained if and only if each $A(i,R_j) = 0$.
\end{lemma}

Assuming Lemma \ref{CloseQ3Q2}, we show how Theorem \ref{DimThree} follows.

\begin{lemma}
\label{A0A1}
Choose integers $R_1 \geq Q_1$, $R_1 < R_2 < n$, and $R_3=Q_3=n$.  Then $c_3(A(0,R_3)-A(1,R_3)) + c_2(A(1,R_2)-A(0,R_2)) + c_1A(0,R_1) \geq 0$.  Also, equality is attained if and only if each $A(i,Q_j) = 0$.
\end{lemma}
\proof We use induction of $R_3-R_2$.  By Lemma \ref{CloseQ3Q2}, the result holds for $R_3 = R_2 - 1$.  Now suppose the result is proven for $R_3-R_2 \leq k-1$, and suppose $R_3-R_2 = k$.  Applying the inductive hypothesis with values $(R_3,R_2+1,R_1)$ and dividing the equation by $c_3$ yields $$(A(0,R_3)-A(1,R_3))+\frac{R_3(R_3-1)(R_3-R_1)}{(R_2+1)R_2(R_2-R_1+1)}(A(1,R_2+1)-A(0,R_2+1))$$ $$+\frac{R_3(R_3-1)(R_3-R_2+1)}{R_1(R_1-1)(R_2-R_1+1)}A(0,R_1) \geq 0.$$  Applying Lemma \ref{CloseQ3Q2} to all induced subcomplexes with $R_2+1$ vertices, using values $(R_2+1,R_2,R_1)$, and averaging the results yields $$(A(0,R_2+1)-A(1,R_2+1))+\frac{(R_2+1)(R_2-R_1+1)}{(R_2-1)(R_2-R_1)}(A(1,R_2)-A(0,R_2))$$ $$+\frac{(R_2+1)R_2}{R_1(R_1-1)(R_2-R_1)}A(0,R_1) \geq 0.$$  Take the appropriate linear combination of these inequalities to obtain the desired result.
\endproof

The multiplicity upper bound conjecture on $\K$ follows by taking $R_1 = Q_1$ and $R_2 = Q_2$ in Lemma \ref{A0A1}.  If $\K$ attains the multiplicity upper bound, then it follows that for all $0 \leq i \leq 1$ and $1 \leq j \leq 3$, $A(i,Q_j)=0$.  Hence, also $A(2,Q_3)=0$, and by Proposition \ref{UPureRes}, $\K$ has a pure resolution.

Our first step in the proof of Lemma \ref{CloseQ3Q2} is the following lemma.

\begin{lemma}
\label{AZero}
Choose integers $2 \leq R_1 < R_2 < R_3 = n$.  Then $c_2A(0,R_2) \geq c_1A(0,R_1) + c_3A(0,R_3)$.  Furthermore, equality occurs only if $A(0,R_3) = A(0,R_2) = A(0,R_1) = 0$.
\end{lemma}
\proof If $G$ is the $1$-skeleton of $\K$ and $2 \leq r \leq n$, then $A(0,r)(G) = A(0,r)(\K)$.  Hence we will prove Lemma \ref{AZero} as a purely graph theoretic result: if $G$ is a graph with $n \geq 4$ vertices, then for integers $2 \leq R_1 < R_2 < R_3 = n$,
\begin{equation}
\label{AZeroMain}
c_2A(0,R_2)(G) \geq c_1A(0,R_1)(G) + c_3A(0,R_3)(G)
\end{equation}
with equality only when each term is zero.

We will prove the result by induction on the number of edges of $G$.  The base case is that $G$ is the complete graph on $n$ vertices, which has ${n(n-1)}/{2}$ edges.  Then $A(0,R_3) = A(0,R_2) = A(0,R_1) = 0$ and the result holds.  Now suppose that for any graph with more than $e$ edges, the lemma holds, and suppose $G$ has $e$ edges, with $e < {n(n-1)}/{2}$.  We will show that $c_2A(0,Q_2) \geq c_1A(0,Q_1) + c_3A(0,Q_3)$ for $G$ in two cases.

Case 1: $G$ is disconnected.  Choose vertices $p$ and $q$ in different components of $G$, and add the edge $pq$ to form $G'$.  Take $A(0,R_j)(G') := A'(0,R_j)$.  If $W \subset V(G)$, then $\tilde{\beta}_0(G'[W]) = \tilde{\beta}_0(G[W])$ if $\{p,q\} \not\subset W$, and $\tilde{\beta}_0(G'[W]) = \tilde{\beta}_0(G[W])-1$ if $\{p,q\} \subset W$.  Hence, $$A'(0,R_3) = A(0,R_3)-1, \quad \quad A'(0,R_2) = A(0,R_2)-\frac{{R_3 - 2 \choose R_2 - 2}}{{R_3 \choose R_2}},$$ $$A'(0,R_1) = A(0,R_1)-\frac{{R_3 - 2 \choose R_1 - 2}}{{R_3 \choose R_1}}.$$  Calculation shows that $$c_2(A(0,R_2) - A'(0,R_2)) = c_3(A(0,R_3) - A'(0,R_3)) + c_1(A(0,R_1) - A'(0,R_1)).$$ Also, $A'(0,R_2)>0$, which we can see by choosing a vertex subset of size $R_2$ that excludes either $p$ or $q$ and includes at least one vertex from each of two components of $G$.  Since $G'$ satisfies Equation (\ref{AZeroMain}) without equality by induction, $G$ also satisfies Equation (\ref{AZeroMain}) without equality.

Case 2: $G$ is connected.  Choose vertices $p$ and $q$ such that $pq$ is not an edge of $G$, and $p$ and $q$ share a common neighbor $v$.  Construct $G'$ from $G$ by adding the edge $pq$ and define $A'(0,R_j) := A(0,R_j)(G')$.  For $i = 1,2$, define $$\mathcal{S}_i := \{W \subset V(G): |W|=R_i, \tilde{\beta}_0(G'[W] = \tilde{\beta}_0(G[W]-1\}.$$ If $W \in \mathcal{S}_2$, then $\{p,q\} \subseteq W$.  If $W \in \mathcal{S}_2$, $W' \subset W$ and $|W'| = R_1$, then $W' \in \mathcal{S}_1$ if and only if $\{p,q\} \subseteq W'$.  For a given $W$, there are ${R_2-2 \choose R_1-2}$ such $W'$.  On the other hand, if $W' \in \mathcal{S}_1$, then $\{p,q\} \subseteq W'$, while $v \not\in W'$.  If $W' \subset W$ and $|W| = R_2$, $W \in \mathcal{S}_2$ only if $v \not\in W$.  For a given $W'$, there are at most $\frac{R_3-R_2}{R_3-R_1}{R_3-R_1 \choose R_2-R_1}$ such $W$.  

It follows that $$\frac{|\mathcal{S}_2|}{|\mathcal{S}_1|} \leq \frac{(R_3-R_1)R_1(R_1-1)}{(R_3-R_2)R_2(R_2-R_1)}\frac{{R_2 \choose R_1}}{{R_3-R_1 \choose R_2-R_1}}.$$ Consequently, $$A(0,R_2)-A'(0,R_2) \leq \frac{(R_3-R_1)R_1(R_1-1)}{(R_3-Q_2)R_2(R_2-R_1)}(A(0,R_1)-A'(0,R_1)),$$ or $c_2(A(0,R_2)-A'(0,R_2)) \leq c_1(A(0,R_1)-A'(0,R_1))$.  $A(0,R_3) = A'(0,R_3) = 0$ since $G$ is connected, and $c_2A'(0,R_2) \geq c_1A'(0,R_1) + c_3A'(0,R_3)$ by the inductive hypothesis.  It follows that $c_2A(0,R_2) \geq c_1A(0,R_1) + c_3A(0,R_3)$ as well.

It remains to be shown that if $c_2A(0,Q_2) = c_3A(0,Q_3) + c_1A(0,Q_1)$, then $A(0,Q_1) = A(0,Q_2) = A(0,Q_3) = 0$.  We assume this fact inductively on $G'$ and consider three cases.

Case 1: $A(0,R_2) = 0$.  Then $c_2A(0,R_2) = c_3A(0,R_3) + c_1A(0,R_1)$ if and only if $A(0,R_3)=A(0,R_2)=A(0,R_1)=0$.

Case 2: $A'(0,R_2) > 0$.  Then $c_2A(0,R_2) > c_3A(0,R_3) + c_1A(0,R_1)$ by the inductive hypothesis on $G'$.

Case 3: $A(0,R_2)(G) > 0$ and $A(0,Q_2)(G') = 0$.  Consider $W \subset V(G), |W| = R_2$ so that $G[W]$ is disconnected.  Then there exists  $W' \subset W, |W'| = R_1$ with either $p \not\in W'$ or $q \not\in W'$ such that $G[W']$ is disconnected.  Then $G'[W']$ is also disconnected.  Hence $A'(0,R_1) > 0$.  It follows by the inductive hypothesis that $c_2A'(0,Q_2) > c_3A'(0,Q_3) + c_1A'(0,Q_1)$ and hence $c_2A(0,Q_2) > c_3A(0,Q_3) + c_1A(0,Q_1)$. \endproof

The following lemma implies that for integers $2 \leq R_1 < R_2 < R_3 \leq n$, the quantity $\sum_{1 \leq i \leq 3,0 \leq j \leq 3}c_{i}(-1)^{i+j}A(j,R_i)$ is determined entirely by $f_3$ and does not depend on $f_2$ or $f_1$.  This fact will be necessary in the proof of Lemma \ref{LooseEdge}.  The proof follows immediately from calculation of $f_3$ in Section \ref{SkipSection}.

\begin{lemma}
\label{HConst}
Let $R_1,R_2,$ and $R_3$ be integers satisfying $2 \leq R_1 < R_2 < R_3 \leq n$.  Then for some positive constant $c$ that depends only on the $R_i$ and not on the $f_i$, $$c\sum_{1 \leq i \leq 3,0 \leq j \leq 3}c_{i}(-1)^{i+j}A(j,R_i) = \frac{n(n-1)(n-2)(n-3)}{R_1R_2R_3}-f_3.$$
\end{lemma}

The next lemma justifies our assumption that no edge of $\K$ is a maximal face.

\begin{lemma}
\label{LooseEdge}
Suppose $\K$ has an edge $pq$ that is a maximal face.  Let integers $R_1 \leq Q_1$, $Q_1 < R_2 < n$, and $R_3=n$ be given, and suppose $\K' = \K-\{pq\}$.  Then $$c_3(A(0,R_3)(\K)-A(1,R_3)(\K)) + c_2(A(1,R_2)(\K)-A(0,R_2)(\K)) + c_1A(0,R_1)(\K) = $$ $$c_3(A(0,R_3)(\K')-A(1,R_3)(\K')) + c_2(A(1,R_2)(\K')-A(0,R_2)(\K')) + c_1A(0,R_1)(\K').$$
\end{lemma}
\proof By Lemma \ref{HConst} and the fact that $f_3(\K) = f_3(\K')$, $$\sum_{1 \leq i \leq 3,0 \leq j \leq 3}c_{i}(-1)^{i+j}A(j,R_i)(\K) = \sum_{1 \leq i \leq 3,0 \leq j \leq 3}c_{i}(-1)^{i+j}A(j,R_i)(\K').$$  For $1 \leq r \leq n$, $A(2,r)(\K) = A(2,r)(\K')$ and $A(3,r)(\K) = A(3,r)(\K')$.  Also, since $A(1,R_1)(\K) = 0$, then $A(1,R_1)(\K') = 0$.  The lemma follows from this. \endproof

Define the \textit{2-components} of $\K$ be the maximal induced subcomplexes of $\K$ that are graph theoretically $2$-connected.

\begin{lemma}
\label{TopDis}
Let $R_3=n$, $R_2 = n-1$, and $2 \leq R_1 \leq Q_1$.  Suppose no maximal face of $\K$ is an edge.  If $c_3A(1,R_3) \geq c_2A(1,R_2)$ and $A(1,R_3) > 0$, then $\K$ contains a $2$-connected component $\K'$ with the following property: $\K'$ contains two vertices, $p$ and $q$, such that $||\K'||-\{p,q\}$ is disconnected.
\end{lemma}
By the assumption that no maximal face of $\K$ is an edge, the conclusion is stronger than that $\K$ is not $3$-connected, and that extra strength will be needed later.  The conclusion can be false if a maximal face of $\K$ is an edge.

The proof of Lemma \ref{TopDis} requires several more technical lemmas.  In the following, we use the fact that $\tilde{H}_1(\K)$ is generated by cycles of the form $[C] = (-1)^{p_1}(v_1 v_2) + \ldots + (-1)^{p_{k-1}}(v_{k-1} v_k) + (-1)^{p_{k}}(v_{k} v_1)$ for vertices $v_1, \ldots, v_k$ and signs $(-1)^{p_i}$ chosen appropriately.  We freely identify $C$ with the graph theoretic cycle $(v_1 v_2 \ldots v_k v_1)$, which is a subcomplex of $\K$.

\begin{lemma}
\label{TopDisHom}
Suppose $\Skel_1(\K)$ is $2$-connected.  Let $p \in V(\K)$ so that $\tilde{\beta}_0(\lk_\K(p)) = s>0$, and label the components of $\lk_\K(p)$ by $\Delta_1,\ldots,\Delta_{s+1}$.  Let $\mathcal{C} = \{C_1,\ldots,C_r\}$ be a set of simple cycles in $\K$ that form a basis for $\tilde{H}_1(\K)$.  Let $C$ be the graph theoretic union $C_1 \cup \ldots \cup C_r$.  Then: \newline
1) For $1 \leq i \leq s+1$, $\Delta_i \cap C \neq \emptyset$. \newline
2) $\tilde{\beta}_1(\K-p) \geq r-s$. \newline
3) If $\tilde{C}$ is a cycle in $\K$ that passes through $p$ exactly once from one component of $\lk_\K(p)$ to another, and $C'$ is a cycle that avoids $p$, then $[\tilde{C}] \neq [C']$ in $\tilde{H}_1(\K)$.
\end{lemma}
\proof Let $W_p$ be a graph with vertices $p \in \K$ and new vertices $p_1,\ldots,p_{s+1}$, and edges $(pp_i)$ for $1 \leq i \leq s+1$.  Construct $\K_p$ by replacing $p$ in $\K$ with $W_p$ so that $\lk_{\K'}(p) = \{p_1,\ldots,p_{s+1}\}$ and for $1 \leq i \leq s+1$, $\lk_{\K'}(p_i) = \Delta_i \coprod \{p\}$.

For $1 \leq i \leq r$, $C_i$ naturally extends to a cycle in $\K_p$ in the following way: if $pv$ is an edge in $C_i$ and $v \in \Delta_i$, replace $pv$ with $pp_iv$.  Since $\K_p$ is homotopy equivalent to $\K$, $(C_1,\ldots,C_r)$ is a basis for $\tilde{H}_1(\K_p)$.  Since $pp_i$ is a maximal edge in $\K_p$, and $\K_p-(pp_i)$ is connected by the hypothesis that $\K$ is $2$-connected, there exists a cycle $C_j \in \mathcal{C}$ such that in $\K_p$, $C_j$ contains $pp_i$.  It follows that $C_j$ in $\K$ contains a vertex of $\Delta_i$, and the first claim holds.  The third claim holds since in $\K_p$, for some $1 \leq i \leq s+1$, $[\tilde{C}]-[C'] \in \tilde{H}_1(\K_p)$ contains the edge $(pp_i)$ with a nonzero coefficient.

The second claim follows from the portion of the Mayer-Vietoris sequence
$$
\ldots \rightarrow \tilde{H}_1(\st_{\K}(p)) \oplus \tilde{H}_1(\K-p) \rightarrow \tilde{H}_1(\K) \rightarrow \tilde{H}_0(\lk_{\K}) \rightarrow \tilde{H}_0(\st_{\K}(p)) \oplus \tilde{H}_0(\K-p) = 0. \Box
$$

If $(C_1,\ldots,C_r)$ is a basis for $\tilde{H}_1(\K)$, call $G = C_1,\ldots,C_r$ a \textit{cycle graph} of $\K$.  Call the operation of replacing $p$ with $W_p$ in the first paragraph of the proof of Lemma \ref{TopDisHom} the \textit{expansion} of $p$ in $\K$.
\begin{lemma}
\label{TopDisCycleGraph}
There exists a cycle graph $G$ of $\K$ such that $\tilde{\beta}_1(G) = \tilde{\beta}_1(\K)$.
\end{lemma}
\proof We show by induction on $m$ that if $1 \leq m \leq r$, there exists a basis $(C_1,\ldots,C_r)$ of $\tilde{H}_1(\K)$ such that $\tilde{\beta}_1(\cup_{i=1}^m C_i) = m$.  The claim is trivial for $m=1$. 

Suppose $(C_1,\ldots,C_{m-1})$ is a set of generators of $\tilde{\beta}_1(\K)$ such that $\tilde{\beta}_1(\cup_{i=1}^{m-1} C_i) = m-1$.  There exists a cycle $C = (v_1,v_2,\ldots,v_t,v_1)$ in $\K$ such that $[C] \not \in $ $([C_1],\ldots,$ $[C_{m-1}])$ $\subset \tilde{H}_1(\K)$.  Decompose $C \cap (\cup_{i=1}^{m-1}C_i)$ into $q$ paths: $(v_{w_i},v_{w_i+1}, \ldots,v_{w_i+l_i})$ for $w_1 < w_2 < \ldots < w_q$ and $l_i$, $1 \leq i \leq q$.

If $q=0$, take $C_m = C$.  In that case $\tilde{\beta}_1(\cup_{i=1}^m C_i) = m$ since $C_m \cap (\cup_{i=1}^{m-1}C_i) = \emptyset$.  If $q=1$, then topologically, $C-\cup_{i=1}^{m-1}$ is an open line segment.  Take $C_m = C$.  It again follows that $\tilde{\beta}_1(\cup_{i=1}^m C_i) = m$.

Otherwise, $q \geq 2$.  Let $P'$ be a path in $\cup_{i=1}^{m-1} C_i$ with endpoints $v_{w_j+l_1}$ and $v_{w_k}$, if such a $P'$ exists, and in that case let $C' = (v_{w_j+l_1}, v_{w_j+l_1+1}, \ldots, v_{w_k})+P'$ (here $+$ denotes concatenation of paths).  If no such $P'$ exists, then $\tilde{\beta}_1(C \cup \cup_{i=1}^{m-1} C_i = m$, and we may take $C_m = C$.  Otherwise, if $[C'] \not \in ([C_1],\ldots,[C_{m-1}])$, then $C' \cap (\cup_{i=1}^{m-1}C_i)$ decomposes into at most $q-1$ paths and the result follows inductively on $q$.  Otherwise if $[C'] \in ([C_1],\ldots,[C_{m-1}])$, then define $\tilde{C}$ by replacing $(v_{w_j+l_1}, v_{w_j+l_1+1}, \ldots, v_{w_k})$ with $P'$ in $C$.  $[\tilde{C}] = [C - C'] \not \in ([C_1],\ldots,[C_{m-1}])$.  $\tilde{C} \cap (\cup_{i=1}^{m-1}C_i)$ decomposes into at most $q-1$ paths, and the result again follows inductively on $q$. \endproof

Let $\phi_G: \tilde{H}_1(G) \rightarrow \tilde{H}_1(\K)$ be the map on homology induced by $i$, the inclusion of $G$ into $\K$.  If $G$ is as in Lemma \ref{TopDisCycleGraph}, then $\phi$ is an isomorphism.

Let $G$ be a graph, and let $W_1,\ldots,W_l$ be a collection of topological closed line segments such that $||G ||= W_1 \cup \ldots \cup W_l$ and that distinct $W_i$ intersect only at their endpoints.  Call $(W_1,\ldots,W_l)$ a \textit{segment decomposition} of $G$.

\begin{lemma}
\label{TopDisSegment}
Let $G$ be a cycle graph arising from the proof of Lemma \ref{TopDisCycleGraph}, and suppose $\tilde{\beta}_1(G) = r$.  Then there exists a segment decomposition of $G$ into $3r-1$ segments.
\end{lemma}
\proof Using $m$ and $C_i$ as in the proof of Lemma \ref{TopDisCycleGraph}, we will prove by induction on $m$ that $\cup_{i=1}^m C_i$ can be decomposed into $3m-1$ segments.  For $m=1$, $C_1$ is a cycle and can be decomposed into two segments.

Suppose $\cup_{i=1}^{m-1}C_i$ can be decomposed into at most $3m-4$ segments.  Take $q$ as in the proof of Lemma \ref{TopDisCycleGraph}.  If $q=0$, then $C_m \cap \cup_{i=1}^{m-1}C_i = \emptyset$ and $\cup_{i=1}^{m}C_i$ can be decomposed into at most $3m-1$ segments.  If $q=1$, then $||C_m|| - \cup_{i=1}^{m-1}||C_i||$ is an open line segment with endpoints $x$ and $y$.  $C_m - \cup_{i=1}^{m-1}C_i$ can itself be taken as one segment, and if $x$ and $y$ are in the interior of segments $W_x$ and $W_y$, then divide $W_x$ and $W_y$ into two segments each.  Then $\cup_{i=1}^{m}C_i$ can be decomposed into $3m-1$ segments.
\endproof

\proofof{Lemma \ref{TopDis}} The hypothesis $c_3A(1,R_3) \geq c_2A(1,R_2)$ is equivalent to $$\tilde{\beta}_1(\K) \geq \frac{c_2}{c_3}A(1,n-1) = \frac{n(n-R_1)}{(n-2)(n-R_1-1)}A(1,n-1).$$  If $\tilde{\beta}_1(\K) = r$, then it follows by $R_1 \geq 2$ that $A(1,n-1) \leq r - 3\frac{r}{n}$.  By applying part 2 of Lemma \ref{TopDisHom} to all $p \in V(\K)$, we assume the weaker hypothesis $\sum_{p \in \K}\tilde{\beta}_0(\lk_\K(p)) \geq 3r$.

First we show that we may assume, without loss of generality, that $\K$ is $2$-connected.  Let $\K_1, \K_2, \ldots, \K_b$ be the $2$-connected components of $\K$.  Choose $p \in V(\K$) so that $\tilde{\beta}_0(\lk_\K(p)) = s$ and $|\{i: 1 \leq i \leq b, p \in V(\K_i)\}| = t$.  The exact Mayer-Vietoris sequence
$$
\ldots \rightarrow \tilde{H}_1(\st_\K(p)) \oplus \tilde{H}_1(\K-p) \rightarrow \tilde{H}_1(\K) \rightarrow \tilde{H}_0(\lk_\K(p)) \rightarrow
$$
$$
\tilde{H}_0(\st_\K(p)) \oplus \tilde{H}_0(\K-p) \rightarrow \tilde{H}_0(\K) \rightarrow 0
$$
implies $\tilde{\beta}_1(\K-p) \geq r-s+t-1$.  Observe $s-t+1 = \sum_{i=1}^b \tilde{\beta}_0(\lk_\K(p) \cap \K_i)$.  Also, $A(1,R_3) = \sum_{i=1}^b \tilde{\beta}_1(\K_i)$.  Hence, for some $1 \leq i \leq b$, $$3\tilde{\beta}_1(\K_i) \leq \sum_{v \in V(\K_i)} \tilde{\beta}_0(\lk_{\K_i}(p))$$ and the result follows inductively on $n$.  Hence, we now assume $\K$ is $2$-connected.

By Lemma \ref{TopDisCycleGraph}, there exists a cycle graph $G$ of $\K$ satisfying $\tilde{H}_1(G) = r$.  By Lemma \ref{TopDisSegment}, there exists a segment decomposition $(W_1,\ldots,W_l)$ of $G$ so that $l < 3r$.  Suppose $p \in V(\K)$ with $\lk_{\K}(p)$ containing $b \geq 2$ components.  Then $p$ is contained in at least $b-1 = \tilde{H}_0(\lk_{\K}(p))$ of the $W_i$.  By the assumption $\sum_{p \in V(\K)} \tilde{\beta}_0(\lk_{\K}(p)) \geq 3r$, there exists a segment $W_i$ for some $1 \leq i \leq l$ such that $W_i$ contains two vertices $p$ and $q$ with $\tilde{H}_0(\lk_{\K}(p)) \neq 0$ and $\tilde{H}_0(\lk_{\K}(q)) \neq 0$.  It follows that $||G|| - \{p,q\}$ is disconnected.  We finish the proof by showing that $||\K|| - \{p,q\}$ is disconnected.

Expand $p$ and $q$ in $G$ and $\K$ to construct $G'$ and $\K'$.  Then $G'$ is a cycle graph of $\K'$ with $\tilde{\beta}_1(G') = \tilde{\beta}_1(\K') = r$.  Consider the path $(p,p',W_i,q',q)$, where $p'$ and $q'$ are the new vertices in $G'$.  By definition of a cycle graph, $G' - \{pp',qq'\}$ is disconnected, and by the Euler-Poincar\'{e} formula, $\tilde{\beta}_1(G' - \{pp',qq'\}) = r-1$.  Since $i^*: \tilde{H}_1(G') \rightarrow \tilde{H}_1(\K')$ is injective and $i^*: \tilde{H}_1(G' - \{pp',qq'\}) \rightarrow \tilde{H}_1(G')$ is injective, then $i^*: \tilde{H}_1(G' - \{pp',qq'\}) \rightarrow \tilde{H}_1(\K' - \{pp',qq'\})$ is also injective.  It follows again by the Euler-Poincar\'{e} formula that $\tilde{\beta}_0((\K' - \{pp',qq'\})) > 0$.  We conclude that $\K'-\{p,q\}$ and hence $||\K||-\{p,q\}$ are disconnected. \endproof

Suppose $\K$ is a $2$-connected simplicial complex with no maximal face an edge and two vertices $p$ and $q$ such that $||\K||-\{p\}$ and $||\K||-\{p\}$ are connected and $||\K||-\{p,q\}$ is disconnected.  Let $\K = R \cup T$ be a union of complexes so that $||R|| \cap ||T|| = \{p,q\}$.  Also suppose that no maximal face of $\K$ is an edge.  Then define a \textit{shifting} operation $S(\K,p,q,R)$ as follows.  First form $\K'$ by replacing $q$ with two vertices $q_1$ and $q_2$ so that the identification of $q_1$ and $q_2$ in $\K'$ gives $\K$, $\lk_{\K'}(q_1) = \lk_\K(q) \cap R$, and $\lk_{\K'}(q_2) = \lk_\K(q) \cap (T)$.  Let $p_1 \in V(\lk_\K(p) \cap R)-\{q_1\}$.  Such a $p_1$ exists because no maximal face of $\K$ is an edge.  Similarly, let $p_2 \in V(\lk_\K(p) \cap T)-\{q_2\}$.  Then $S(\K,p,q,R)$ is $\K'$ with $p_1$ and $p_2$ identified.  Observe that shifting preserves $f_3$ and $n$.

Our next two lemmas allow us to use shifting as an inductive tool in proving Lemma \ref{CloseQ3Q2}.

\begin{lemma}
\label{Shifting}
Let $\K$ be as above.  Then shifting preserves $c_3(A(0,R_3)-A(1,R_3)) + c_2(A(1,R_2)-A(0,R_2)) + c_1A(0,R_1)$.
\end{lemma}
\proof We will show that the following quantities are preserved by shifting: $A(p,r)$ when $p \geq 2$, and $A(1,R_1)$.  From this, it follows from Lemma \ref{HConst} and the observation that shifting preserves $f_3$ that shifting then preserves $c_3(A(0,R_3)-A(1,R_3)) + c_2(A(1,R_2)-A(0,R_2)) + c_1A(0,R_1)$.

Let $R$ and $T$ be as above, and let $\K' = S(\K,p,q,R)$ be a shift of $\K$.  To show $A(p,r)(\K) = A(p,r)(\K')$ for $2 \leq r \leq n$ and $p \geq 2$, observe that $$A(p,r)(\K) = \frac{1}{{n \choose r}}\sum_{|W|=r}\tilde\beta_p(\K[W]) = \frac{1}{{n \choose r}}\sum_{|W|=r}(\tilde\beta_p(\K[W]\cap R)+\tilde\beta_p(\K[W]\cap T))$$ since $\tilde{H}_i(\K[W]\cap R \cap T) = 0$ for $i \geq 1$.  The same calculation hold for $\K'$, and so $A(p,r)(\K) = A(p,r)(\K')$.

To prove $A(1,R_1)$ is preserved under shifting, observe that $\K' = R \cup T$ and that $R \cap T$ in $\K'$ is an edge.  Choose $W \subset V(\K')$ of minimal size so that $\tilde{H}_1(\K'[W]) \neq 0$.  $\K'[W]\cap R' \cap T'$ is either empty or connected, so $\tilde{H}_1(\K'[W]) = \tilde{H}_1(\K'[W]\cap T') \oplus \tilde{H}_1(\K'[W]\cap R')$.  By minimality of $W$, $W \subset R'$ or $W \subset T'$.  Then there is a corresponding $W_1 \subset V(\K)$ so $\tilde{H}_1(\K[W_1]) \neq 0$.  By hypothesis, $|W_1| > R_1$, so $|W| > R_1$ and we conclude $A(1,R_1)(\K') = 0$. \endproof

\begin{lemma}
\label{LimShift}
Let $\K'$ be a shift of $\K$.  Then $A(1,n)(\K') < A(1,n)(\K)$.
\end{lemma}
\proof If $\K$ is disconnected, restrict attention to the component of $\K$ in which the shift occurs.  Let $R$ and $T$ be as above.  For $\K$, there is a Mayer-Vietoris sequence 
$$\ldots \rightarrow \tilde{H}_1(R \cap T) \rightarrow \tilde{H}_1(R)\oplus\tilde{H}_1(T)\rightarrow \tilde{H}_1(\K) \rightarrow \tilde{H}_0(R \cap T) \rightarrow \tilde{H}_0(R) \oplus \tilde{H}_0(T) \rightarrow \ldots$$
Since $\tilde{H}_1(R \cap T) = 0$ and $R$ and $T$ are connected, this reduces to
$$0 \rightarrow \tilde{H}_1(R)\oplus\tilde{H}_1(T)\rightarrow \tilde{H}_1(R \cup T) \rightarrow \tilde{H}_0(R \cap T) \rightarrow 0.$$
Thus, $\tilde{\beta}_1(\K) = \tilde{\beta}_1(R)+\tilde{\beta}_1(T)+1$.  For $\K'$, there is a similar Mayer-Vietoris sequence 
$$0 \rightarrow \tilde{H}_1(R)\oplus\tilde{H}_1(T)\rightarrow \tilde{H}_1(R \cup T) \rightarrow \tilde{H}_0(R \cap T) \rightarrow 0.$$
Since $R \cap T$ is connected, $\tilde{\beta}_1(\K') = \tilde{\beta}_1(R)+\tilde{\beta}_1(T) = \tilde{\beta}_1(\K)-1$. \endproof

\proofof{Lemma \ref{CloseQ3Q2}} By Lemma \ref{LooseEdge}, assume no maximal face of $\K$ is an edge.  The quantity $c_3(A(0,R_3)-A(1,R_3)) + c_2(A(1,R_2)-A(0,R_2)) + c_1A(0,R_1)$ is preserved under shifting, so if $\K$ admits a shift, construct $\K'$ by shifting $\K$ repeatedly until $\K$ does not admit any more shifting.  By Lemma \ref{LimShift}, this occurs after a finite number of steps.  Since $\K'$ does not admit shifting, by Lemma \ref{TopDis}, $c_3A(1,R_3)(\K') \leq c_2A(1,R_2)(\K')$.  The inequality $c_3(A(0,R_3)-A(1,R_3)) + c_2(A(1,R_2)-A(0,R_2)) + c_1A(0,R_1) \geq 0$ follows from this and Lemma \ref{AZero}.  Furthermore, equality is attained only if each term is zero. \endproof

This completes the proof of Theorem \ref{DimThree}.

\section{Multiplicity Upper Bound on Homology Manifolds with Many Vertices}
\label{HomManifold}

A $(d-1)$-dimensional simplicial complex $\K$ is Gorenstein* over $\field$ (also called a \textit{homology sphere} over $\field$) if for every face $F \in \K$, $\tilde{H}_{d-1-|F|}(\lk_\K(F)) = \field$ and $\tilde{H}_{i}(\lk_\K(F)) = 0$ for $i < d-1-|F|$.  $\K$ is a \textit{homology manifold} over $\field$ if the previous conditions hold for all $F \neq \emptyset$.  If $\K$ is a connected homology manifold, then $\tilde{H}_{d-1}(\K)$ is either $0$ or $\field$.  In the latter case, we say that $\K$ is \textit{orientable} over $\field$.  The class of homology manifolds is an extension of the class of triangulations of topological manifolds.

Let $\K$ be a $(d-1)$-dimensional homology manifold with $n$ vertices and Euler characteristic $\chi = \beta_{0}(\K) - \beta_{1}(\K) + \ldots + (-1)^{d-1} \beta_{d-1}(\K)$ (here we used non-reduced Betti numbers).  The main result of this section is that if $d$ is odd and if $n$ is sufficiently large relative to $d$ and $|\chi|$, then $\K$ satisfies the multiplicity upper bound conjecture.  Furthermore, if $d \leq 10$, $\K$ is orientable over $\field$, and $n$ is large, then $\K$ satisfies the multiplicity upper bound conjecture.

First we need the following lemma.
\begin{lemma}
\label{MLink}
Let $\K$ be a $(d-1)$-dimensional simplicial complex with $n$ vertices.  Let $F$ be a face of $\K$ so that $\lk_\K(F)$ has dimension $d'$ and $n'$ vertices.  Then for $1 \leq i \leq n'-d'$, $M_i(\lk_\K(F)) \leq M_i(\K)$.
\end{lemma}
\proof Let $F = \{v_1,\ldots,v_r\}$.  Let $\K_i = \lk_\K(v_1 \ldots v_i)$ and $\K_0 = \K$, so that for $1 \leq i \leq r$, $\K_i = \lk_{\K_{i-1}}(v_i)$.  By induction on $r$, we assume without loss of generality that $F$ is a single vertex $v$.

Let $i$ be given, and suppose that $M_i(\lk_\K(F)) = k$, so that there exists $W \subseteq \lk_\K(v)$ with $|W| = k$ such that $\Delta' := \lk_\K(v)[W]$ satisfies $\tilde{H}_{k-i-1}(\Delta') \neq 0$.  Let $\Delta := \K[W]$.  If $\tilde{H}_{k-i-1}(\Delta) \neq 0$, then $M_i(\K) \geq |W| = M_i(\lk_\K(v))$ as desired.  So now suppose $\tilde{H}_{k-i-1}(\Delta) = 0$.

Let $\Delta'' := \K[W \cup v]$, and let $\Delta_1$ be the $v$-cone over $\Delta'$.  Then we can write $\Delta'' = \Delta_1 \cup \Delta$.  Also $\Delta_1 \cap \Delta = \Delta'$, and we have the following exact sequence, which is part of the Mayer-Vietoris sequence on homology: 
$$\tilde{H}_{k-i}(\Delta'') \rightarrow \tilde{H}_{k-i-1}(\Delta') \rightarrow \tilde{H}_{k-i-1}(\Delta_1)\oplus \tilde{H}_{k-i-1}(\Delta).
$$
$\Delta_1$ is a cone, and hence $\tilde{H}_{k-i-1}(\Delta_1) = 0$.  Also, $\tilde{H}_{k-i-1}(\Delta) = 0$ and $\tilde{H}_{k-i-1}(\Delta') \neq 0$ by hypothesis.  Hence $\tilde{H}_{k-i}(\Delta'') \neq 0$.  Since $\Delta''$ is a subcomplex of $\K$ induced on $k+1$ vertices, then $M_i(\K) \geq k+1 > M_i(\lk_\K(v))$, which proves the result. \endproof

Suppose $\tilde{H}_{d'-1}(\lk_\K(F)) \neq 0$.  Then Lemma \ref{MLink} implies that $Q_i(\K) \leq Q_i(\lk_\K(F))$ for $1 \leq i \leq d'-1$.  We state a natural analog of Lemma \ref{MLink} for minimal shifts.  The proof is similar and will be omitted.

\begin{lemma}
Let $\K$ be a $(d-1)$-dimensional simplicial complex with $n$ vertices.  Let $F$ be an $(r-1)$-face of $\K$, so that $\lk_\K(F)$ has dimension $d'$ and $n'$ vertices.  Then for $1 \leq i \leq n'-d'$, $m_i(\lk_\K(F)) \geq m_i(\K)-r$.
\end{lemma}

For the rest of this section, let $\K$ be a $(d-1)$-dimensional homology manifold with $n$ vertices and Euler characteristic $\chi$.  Assume that $n$ is sufficiently large relative to $d$ and $|\chi|$.  The next lemma is an important ingredient in our calculations.

\begin{lemma}
\label{Qbound}
There exists a function $z_d(k)$, independent of $\K$ and $n$, such that $Q_k(\K) \leq z_d(k)$ for all $1 \leq k \leq \lfloor \frac{d-1}{2}. \rfloor$
\end{lemma}
\proof Choose $1 \leq k < \lfloor \frac{d-1}{2} \rfloor$, and let $F \in \K$ be a face of dimension $d - 2k - 2$.  By definition of a homology manifold, $\lk_\K(F)$ is a homology sphere of dimension $2k$.  By Lemma \ref{MLink} and the fact that $\tilde{H}_{2k}(\lk_\K(F)) \neq 0$, it suffices to prove that $Q_k(\lk_\K(F)) \leq z_d(k)$.

Use the Dehn-Sommerville equation $h_k = h_{k+1}$ on $\lk_\K(F)$ and the formula $h_k = \sum_{i=0}^k {d-i \choose d-k}f_{i-1}$ to solve for $f_k(\lk_\K(F))$ as a linear combination of $f_{-1}(\lk_\K(F))$, $f_0(\lk_\K(F))$, $\ldots,f_{k-1}(\lk_\K(F))$.  Since $k < \lfloor \frac{\dim \lk_\K(F)+1}{2} \rfloor$, we have $f_{-1}(\lk_\K(F)) < f_0(\lk_\K(F)) < \ldots < f_{k-1}(\lk_\K(F))$.  These inequalities follow from the fact that each $i$-face contains $(i+1)$ faces of dimension $(i-1)$, and each $(i-1)$-face is contained in at least $(2k+1-i)$ $i$-faces.  This yields $f_k(\lk_\K(F)) \leq z'(k)f_{k-1}(\lk_\K(F))$ for some function $z'$.  Since each $k$-face contains $k$ $(k-1)$-faces, there exists a $(k-1)$-face $F'$ of $\lk_\K(F)$ that is contained in no more than $kz'(k)$ $k$-faces.  Then $\lk_{\K}(F \cup F')$ is a homology sphere of dimension $k$ with at most $kz'(k)$ vertices.  Another application of Lemma \ref{MLink} yields $Q_k(\K) \leq Q_k(\lk_\K(F)) \leq kz'(k)$.  

If $k = (d-1)/2$, we have the following Dehn-Sommerville equation for homology manifolds \cite{Klee}: 
$$h_{d-k} - h_k = (-1)^k {d \choose k}(\chi - (1+(-1)^{d-1})).$$
Since $n$ is large relative to $\chi$, the term $(-1)^k {d \choose i}(\chi - (1+(-1)^{d-1}))$ is small relative to $f_{d-k-1}(\K)$ and the reasoning of the above paragraph applies. \endproof

Note that the assumption for $\K$ to be orientable over $\field$ is unnecessary for Lemma \ref{Qbound}.  The critical property is that $\lk_\K(F)$ is orientable when $\dim F \geq \lfloor \frac{d-2}{2} \rfloor$.  This is always the case for homology manifolds.

Now we are ready to prove the main result when $d$ is odd.

\begin{theorem}
\label{HomManOddD}
Let $\K$ be a $(d-1)$-dimensional homology manifold with $d$ odd, Euler characteristic $\chi$, and $n$ vertices.  Suppose $n$ is large relative to $d$ and $|\chi|$.  Then $\K$ satisfies the multiplicity upper bound conjecture without equality.
\end{theorem}
\proof For all simplicial complexes, $h_i$ is bounded above by a polynomial of degree $i$ in $n$.  Using $h_{d-i} - h_i = (-1)^i {d \choose i}(\chi - (1+(-1)^{d-1}))$ and the fact that $\chi$ is small relative to $n$, we have that $f_{d-1} = \sum_{i=0}^d h_i$ is also bounded above by a degree $(d-1)/2$ polynomial in $n$.

By Lemma \ref{Qbound}, $Q_i \leq z_d(i)$ for some function $z_d$ and $0 \leq i \leq (d-1)/2$.  Hence, it suffices to show that
$$
f_{d-1} < \frac{\prod_{i=0}^{d-1} (n-i)}{\prod_{i=1}^{d-1}Q_i} \leq \frac{\prod_{i=0}^{d-1} (n-i)}{\prod_{i=1}^{\frac{d-1}{2}}z_d(i) \prod_{i=\frac{d+1}{2}}^{d-1}Q_i}.
$$
The latter term is bounded below by a polynomial in $n$ of degree $(d+1)/2$, which proves the result. \endproof

The condition that $n$ is large is necessary for the assertion that $\K$ does not attain the upper bound.  For example, the boundary of the cross polytope of dimension $d$ has $2d$ vertices and it has a pure resolution.

For the remainder of this section, we will assume that $\K$ is orientable over the base field $\field$.

\begin{theorem}
\label{EvenD}
Let $\K$ be as above, where $d \leq 10$.  Then $\K$ satisfies the multiplicity upper bound conjecture.  Furthermore, $\K$ attains the upper bound if and only if $\K$ is Gorenstein* and $(d/2)$-neighborly.
\end{theorem}
Before proving this result, we need the following weakenings of the multiplicity conjecture for $d \leq 5$.
\begin{lemma}
\label{NearUBI3}
Let $\Delta$ be a $(d-1)$-dimensional simplicial complex with $d \leq 5$ and $n$ vertices.  Then $$f_{d-1}(\Delta) \leq \frac{\prod_{i=0}^{d-1} (n-i)}{d\prod_{i=1}^{d-2} Q_i}.$$
\end{lemma}
Lemma \ref{NearUBI3} differs from the multiplicity upper bound conjecture in that $Q_{d-1}$ is replaced by $d$.  Since the maximal shift sequence of $\Delta$ is strictly increasing, $Q_{d-1} \geq d$ is always satisfied, and Lemma \ref{NearUBI3} is weaker than the multiplicity upper bound conjecture.

\proofof{Lemma \ref{NearUBI3}} The multiplicity upper bound conjecture is known to hold for $d \leq 4$ (Theorem \ref{DimThree}), so we restrict our attention to $d=5$.  First consider the case that $\tilde{H}_4(\K) = 0$.  Then for each $v \in V(\K)$,
\begin{equation}
\label{NearUBI3Equation}
f_3(\lk_\K(v)) \leq \frac{\prod_{i=1}^{\codim \lk_\K(v)} M_i(\lk_\K(v))}{(\codim \lk_\K(v))!} \leq \frac{\prod_{i=1}^{n-d} M_i(\K)}{(n-d)!}.
\end{equation}
The first inequality follows by the multiplicity upper bound conjecture for $d=4$ and the second inequality follows by Lemma \ref{MLink} and the fact that $\codim \lk_\K(v) \leq n-d$ for all $v \in V(\K)$.  Since each $4$-face of $\K$ contains $5$ vertices, summing the inequalities of Equation (\ref{NearUBI3Equation}) over all $v \in V(\K)$ yields the result.

Now consider the case $\tilde{H}(\K) \neq 0$.  The result follows by a variant of the calculation in Equation (\ref{HomRedEquation}).
\endproof

\begin{lemma}
\label{TopQ}
Let $\Delta$ be a $(d-1)$-dimensional simplicial complex with $n$ vertices.  Then $$f_{d-1}(\Delta) \leq \frac{d}{Q_{d-1}(\Delta)}{n \choose d}.$$
\end{lemma}
Lemma \ref{TopQ} differs from the multiplicity upper bound conjecture in that $Q_i$ is replaced by $i+1$ for $1 \leq i \leq d-2$.

\proofof{Lemma \ref{TopQ}} Suppose $W \subset V(\Delta)$ and $|W| = Q_{d-1}$.  Then by definition of $Q_{d-1}$, $\tilde{H}_{d-1}(\Delta[W]) = 0$.  By the main result of \cite{BK}, $f_{d-1}(\Delta[W]) \leq \frac{d}{Q_{d-1}}{Q_{d-1} \choose d}$.  Adding over all $W \subset V(\Delta)$ with $|W| = Q_{d-1}$ yields $f_{d-1}(\Delta) \leq \frac{d}{Q_{d-1}}{n \choose d}$.
$\Box$

We also need the following refinement to Lemma \ref{NearUBI3}.
\begin{lemma}
\label{NearerUBI3}
Suppose the conditions of Lemma \ref{NearUBI3} hold, and in addition $Q_{d-1} > d$.  Let $R_i, 1 \leq i \leq d-1$ be arbitrary positive integers satisfying $R_i \leq Q_i$.  There exists a $\xi > 0$ which depends only on $d$ and the $R_i$ such that $$f_{d-1} \leq (1-\xi)\frac{\prod_{i=0}^{d-1} (n-i)}{d\prod_{i=1}^{d-2} R_i}$$ if $n$ is sufficiently large.
\end{lemma}
Note that the minimum choice for $n$ might depend on $\xi$.

\proof If suffices to take $R_i = Q_i$ for $1 \leq i \leq d-1$.  As with Lemma \ref{NearUBI3}, Lemma \ref{NearerUBI3} holds for $d \leq 4$ since the multiplicity upper bound conjecture is known to hold in this case.  We consider $d=5$.  If $Q_4 > 5$ and $Q_3 = 4$, then the result follows from Lemma \ref{TopQ} by taking $\xi \leq \frac{Q_4-5}{Q_4}$.  Consider $Q_3 > 4$.  There exists an integer $n_0 > Q_4$ so that $\frac{n_0(n_0-1)(n_0-2)(n_0-3)(n_0-4)}{5Q_1Q_2Q_3}$ is not an integer.  Then for $W \subset V(\Delta)$ and $|W| = n_0$, $\Delta[W]$ is a complex with $n_0$ vertices, dimension at most $4$, and upper skips $Q_i$ or greater.  By Lemma \ref{NearUBI3},  
$$f_4(\Delta[W]) < \frac{n_0(n_0-1)(n_0-2)(n_0-3)(n_0-4)}{5Q_1Q_2Q_3},$$
and so
$$f_4(\Delta[W]) \leq (1 - \xi) \frac{n_0(n_0-1)(n_0-2)(n_0-3)(n_0-4)}{5Q_1Q_2Q_3}$$
for some $\xi > 0$.  If $n > n_0$, sum over all $W \subset V(\Delta)$ with $|W| = n_0$ to obtain the result. 
\endproof

We also need the upper bound theorem for homology manifolds, which is Theorem 1.4 of \cite{NovUHM}.
\begin{theorem}
\label{UHM}
If $d$ is even and $\K$ is a $(d-1)$-dimensional homology manifold, then 
\begin{equation}
\label{UHMEquation}
f_{d-1}(\K) \leq \frac{n(n-\frac{d}{2}-1)(n-\frac{d}{2}-2)\ldots(n-d+1)}{(d/2)!}-
\end{equation}
Furthermore, equality is attained only if $\K$ is $(d/2)$-neighborly.
\end{theorem}

The right-hand side is the number of facets of a $(d/2)$-neighborly Gorenstein* complex.  Hence $f_{d-1}(\K) = \frac{n(n-\frac{d}{2}-1)(n-\frac{d}{2}-2)\ldots(n-d+1)}{(d/2)!}$ only if $\beta_i(\K) = 0$ for $0 \leq i < d/2-1$.  If $\beta_{d/2-1}(\K) = 0$, then $\K$ is Gorenstein*.

We are now ready to prove Theorem \ref{EvenD}.  The case of odd $d$ is already proven.  Our proof will restrict to $d=10$; the case $d<10$ and $d$ even are similar.

\proofof{Theorem \ref{EvenD}} Since $\K$ is orientable, $Q_9 \neq n$, and we have $Q_i \leq n-10+i$ for $5 \leq i \leq 9$.  First consider the case that $Q_i = i+1$ for $1 \leq 1 \leq 4$.  Then it suffices to show $f_9 \leq \frac{n(n-1)\ldots(n-9)}{Q_1\ldots Q_9}$, or $f_9 \leq \frac{n(n-6)\ldots(n-9)}{5!}.$  This is true by Theorem \ref{UHM}, and equality is treated in Lemma \ref{HomManEquality}.  Henceforth we will assume $Q_4 > 5$.

By Lemma \ref{Qbound}, for $1 \leq i \leq 4$, $Q_i \leq z_{10}(i)$, and hence $$U(\K) \geq \frac{(n-5)(n-6)(n-7)(n-8)(n-9)}{z_{10}(1)z_{10}(2)z_{10}(3)z_{10}(4)}.$$  We assume $$f_9(\K) \geq \frac{(n-5)(n-6)(n-7)(n-8)(n-9)}{z_{10}(1)z_{10}(2)z_{10}(3)z_{10}(4)}.$$  From the Dehn-Sommerville equations, there exists constants $c_i$ for $-1 \leq i \leq 4$ such that $f_9(\K) = \sum_{i=-1}^4 c_if_i$.  For $\epsilon > 0$, if $n$ is sufficiently large, then $f_i < \epsilon n^5$ for $i \leq 3$.  It follows that there exists a $c > 0$, independent of $n$ such that $f_4 \geq cn^5$ provided $n$ is sufficiently large.

For $\epsilon > 0$ and $n$ large relative to $\frac{1}{\epsilon}$, $f_{i} < \epsilon f_4$ when $-1 \leq i \leq 3$.  Hence $h_5 \approx f_4$ and $h_i < \epsilon f_4$ for $i < 5$.  By the Dehn-Sommerville equations, $h_i < \epsilon f_4$ for $i>5$ as well.  Then we have 
\begin{equation}
\label{HM1}
(1-\epsilon)f_4 < f_9 < (1+\epsilon)f_4
\end{equation}
for $\epsilon > 0$, and $n$ sufficiently large relative to $\frac{1}{\epsilon}$.

The maximal shifts are nonincreasing under taking induced subcomplexes.  If $\K$ and $\K[W]$ for $W \subset V(\K)$ have $i$th upper skips $Q_i$ and $Q_i'$ respectively, then $Q_i \leq Q_i'$ unless $|W| \leq Q_i-d+i$.  Let $v$ be the larger of $Q_4+10$ and $n_0$ as constructed in the proof of Lemma \ref{NearerUBI3} with values $d=5$ and $R_i = Q_i$ for $i \leq 4$.  Suppose $W \subset V(\K)$ with $|W| = v$ and $\K[W]$ has dimension at most $4$.  Then by Lemma \ref{NearerUBI3}, taking $\xi$ to be the value obtained from Lemma \ref{NearerUBI3} with values $d=5$ and $R_i = Q_i$ for $i \leq 4$,
\begin{equation}
\label{HM2}
f_4(\K[W]) \leq (1 - \xi)\frac{v(v-1)(v-2)(v-3)(v-4)}{5Q_1Q_2Q_3}.
\end{equation}
For an arbitrary fixed $\delta > 0$, (again is $n$ large relative to $\frac{1}{\delta}$), $f_5 < \delta {n \choose 6}$ since by the Dehn-Sommerville equations, $f_5$ is bounded above by a polynomial in $n$ of degree $5$.  In other words, a fraction less than $\delta$ of all sets of $6$ vertices in $\K$ form a $5$-face.  Over all $W \subset V(\K)$ with $|W| = v$, the average value of $f_5(\K[W])$ is $\delta {v \choose 6}$.  If $\delta < 1/{v \choose 6}$, then of the ${n \choose v}$-subsets $W \subset V(\K)$ with $|W| = v$, fewer than $\delta{v \choose 6}{n \choose v}$ contain a $5$-face.  If $\K[W]$ contains a $5$-face, then trivially $f_4(\K[W]) \leq {v \choose 5}$.  Adding over all such $\K[W]$, it follows from (\ref{HM2}) that
$$
f_4(\K) \leq (1-\delta{v \choose 6})(1 - \xi)\frac{n(n-1)(n-2)(n-3)(n-4)}{5Q_1Q_2Q_3} + \delta{v \choose 6}{n \choose 5}.
$$
By choosing $n$ large enough so that $\delta$ is sufficiently small relative to $\xi$, it follows that $f_4 < (1 - \frac{\xi}{2})\frac{4!}{Q_1Q_2Q_3}{n \choose 5}$.  This, together with (\ref{HM1}), yields 
\begin{equation}
\label{HM3}
f_9 < (1 - \frac{\xi}{2})(1+\epsilon)\frac{4!}{Q_1Q_2Q_3}{n \choose 5}.
\end{equation}

By applying similar reasoning as above and using Lemma \ref{TopQ}, for $\gamma > 0$ and $n$ sufficiently large relative to $\frac{1}{\gamma}$, we have $f_4 < (1+\gamma){n \choose 5}\frac{5}{Q_4}$.  Alternately, 
\begin{equation}
\label{HM4}
\frac{f_4}{f_3}\frac{f_3}{f_2}\frac{f_2}{f_1}\frac{f_1}{f_0}\frac{f_0}{f_{-1}} < (1+\gamma)\frac{n(n-1)(n-2)(n-3)(n-4)}{5!}\frac{5}{Q_4}.
\end{equation}

For $i \geq 0$, the link of an $i$-face is a homology sphere of dimension $8-i$.  Since every $(9-i)$-face contains $9-i+1$ $(9-i-1)$-faces, there exists a $(9-i-1)$-face that is contained in at most $(9-i+1)\frac{f_{9-i}}{f_{9-i-1}}$ faces of dimension $9-i$.  By the remarks following Lemma \ref{MLink}, the $Q_i$ are nonincreasing under links, and it follows that $Q_i < (9-i+1)\frac{f_{9-i}}{f_{9-i-1}}$ for $i<9$.  Combining this with Equation (\ref{HM4}) yields 
\begin{equation}
\label{HM5}
Q_5\ldots Q_9 < 5!\frac{f_4}{f_3}\frac{f_3}{f_2}\frac{f_2}{f_1}\frac{f_1}{f_0}\frac{f_0}{f_{-1}} < (1+\gamma)n(n-1)(n-2)(n-3)(n-4)\frac{5}{Q_4}.
\end{equation}

From Equation (\ref{HM3}), and by taking $\gamma$ and $\epsilon$ sufficiently small relative to $\xi$, we find 
$$
f_9 < (1 - \frac{\xi}{2})(1+\epsilon)\frac{n(n-1)\ldots(n-4)}{Q_1Q_2Q_35} \leq \frac{n(n-1)\ldots(n-9)}{Q_1Q_2Q_3(1+\gamma)n(n-1)\ldots(n-4)5}.
$$
Finally, Equation (\ref{HM5}) yields $f_9 < \frac{n(n-1)\ldots(n-9)}{Q_1\ldots Q_9}$, as desired. By choosing $\epsilon$ sufficiently small and $n$ large, we find that $\K$ does not attain the multiplicity upper bound if $Q_4 > 5$. \endproof

\begin{lemma}
\label{HomManEquality}
Let $\K$ be as above.  Then $\K$ attains the multiplicity upper bound only if $\K$ is Gorenstein* and $d/2$-neighborly.
\end{lemma}
\proof
If $\K$ attains the multiplicity upper bound, then by $$U(\K) \geq \frac{n(n-\frac{d}{2}-1)(n-\frac{d}{2}-2)\ldots(n-d+1)}{(d/2)!}$$ and Lemma \ref{UHM}, $\K$ is $d/2$-neighborly and hence $\beta_i(\K) = 0$ for $0 \leq i < d/2-1$.  By Poincar\'e's duality, it suffices to show $\beta_{d/2-1}(\K) = 0$, which implies that $\beta_i(\K) = 0$ for $0 \leq i \leq d-1$ and that $\K$ is Gorenstein*.  We show that if $\beta_{d/2-1}(\K) \neq 0$, then $M_{n-d-1}(\K) \geq n-d/2$ and hence $U(\K) > \frac{n(n-\frac{d}{2}-1)(n-\frac{d}{2}-2)\ldots(n-d+1)}{(d/2)!}$, which contradicts the assumption that $f_{d-1}(\K) = U(\K)$.

Choose $W \subset V(\K)$ with $|W| = d/2$.  Then $\K[W]$ is a simplex.  Using the long exact cohomology sequence of a pair, we obtain $H^p(\K,\K[W]) = H^p(\K)$ for all $p>0$.  By Lefschetz Duality (see Section 70 of \cite{Munkres}), $H^p(\K,\K[W]) \approx \tilde{H}_{d-1-p}(||\K|| - ||\K[W]||)$.  Since $\K[W]$ is an induced subcomplex of $\K$, $||\K[V(\K) - W]||$ is a deformation retract of $||\K||-||\K[W]||$, and hence $\tilde{H}_{d-1-p}(||\K|| - ||\K[W]||) \approx \tilde{H}_{d-1-p}(\K[V(\K) - W])$.  For $p = d/2-1$, we have $$0 \neq H^{\frac{d}{2}-1}(\K) \approx H^{\frac{d}{2}-1}(\K,\K[W]) \approx \tilde{H}_{(d-1)-(\frac{d}{2}-1)}(||\K||-||\K[W]||) \approx \tilde{H}_{\frac{d}{2}}(\K[V(\K)-W).$$  We conclude that $M_{n-d-1}(\K) \geq n-d/2$.
\endproof

By the reasoning of the above proofs, if the multiplicity upper bound conjecture is proven on simplicial complexes of dimension up to $r-1$, then Theorem \ref{EvenD} follows for $d \leq 2r+2$.

\section {Multiplicity Lower Bound on Complexes with Many Vertices}
\label{LargeLowerBound}
In this section we look at the multiplicity lower bound conjecture for complexes with many vertices.  If a Cohen-Macaulay $\K$ has many vertices and attains the multiplicity lower bound, then our main theorem places restrictions on the structure of $\K$.
\begin{theorem}
\label{LLB}
For every positive integer $d$ and $\epsilon > 0$, there exists an integer $n_0 = n_0(d,\epsilon)$ such that the following holds: if $\K$ is a $(d-1)$-dimensional Cohen-Macaulay complex with $n > n_0$ vertices that attains the multiplicity lower bound, and $r$ is the smallest integer such that $\K$ is not $r$-neighborly, then \newline
1) ${n-d+r-1 \choose r-1} \leq f_{d-1} \leq (1+\epsilon){n-d+r-1 \choose r-1}$, \newline
2) $h_i(\K) < \epsilon{n \choose r-1}$ for $i \geq r$.
\end{theorem}

\proof Fix a sequence of small positive real numbers $\epsilon_0 < \epsilon_1 < \ldots < \epsilon_d$ that depend only on $d$.  Choose $t$ to be the smallest integer such that $f_t < (1-\epsilon_t){n \choose t+1}$.  If no such $t$ exists, take $t = d$.  By hypothesis, $t \geq r-1$.  Then $$\frac{f_t}{f_{t-1}} < \frac{1-\epsilon_t}{1-\epsilon_{t-1}}\frac{{n \choose t+1}}{{n \choose t}} = \frac{1-\epsilon_t}{1-\epsilon_{t-1}}\frac{n-t}{t+1}.$$ Since every $t$-face contains $t+1$ faces of dimension $t-1$, there exists a $(t-1)$-face contained in at most $\frac{1-\epsilon_t}{1-\epsilon_{t-1}}(n-t)$ $t$-faces.  The Cohen-Macaulay connectivity is nondecreasing under taking links \cite{Baclawski}, and so $q_t(\K) < \frac{1-\epsilon_t}{1-\epsilon_{t-1}}(n-t)$.  Also since the Cohen-Macaulay connectivity sequence of $\K$ is strictly decreasing, $$q_i < \frac{1-\epsilon_t}{1-\epsilon_{t-1}}(n-t) \quad \mbox{for} \quad i \geq t.$$

The above formula implies that $L(\K) = \frac{n(n-1)\ldots(n-d+1)}{(n-q_0+1)\ldots(n-q_{d-1}+1)}$ is bounded above by a degree $t$ polynomial in $n$.  If, for some small $\nu > 0$ independent of $n$, $f_t \geq \nu {n \choose t+1}$, then $h_{t+1} > \delta{n \choose t+1}$ for some $\delta > 0$ independent of $n$, and since the $h$-vector of a Cohen-Macaulay complex is nonnegative, $f_{d-1} > \delta{n \choose t+1}$.  This contradicts the hypothesis that $\K$ attains the multiplicity lower bound, so in fact $f_t < \nu {n \choose t+1}$.

Using that $$\frac{f_t}{f_{t-1}} < \frac{\nu}{1-\epsilon_{t-1}}\frac{{n \choose t+1}}{{n \choose t}}, \quad \mbox{we obtain}$$ $$q_t < \frac{\nu}{1-\epsilon_{t-1}}(n-t).$$ Hence, by taking $q_i \leq n-i-1$ for $i < t$ and taking $\nu$ sufficiently small, $$L(\K) = \frac{n(n-1)\ldots(n-d+1)}{(n-q_0+1)\ldots(n-q_{d-1}+1)} <(1+\epsilon){n-d+t \choose t}.$$

That $f_{t-1} > (1-\epsilon_{t-1}){n \choose t}$ implies that, for some $\delta_{t-1} \rightarrow 0$ as $\epsilon_{t-1} \rightarrow 0$, $h_t > (1-\delta_{t-1}){n \choose t}$.  Hence $f_{d-1} > (1-\delta_{t-1}){n \choose t}$.  By choosing $\epsilon_t$ small enough, $f_{d-1} > (1-\epsilon){n \choose t}$.

Since by our assumption $f_{d-1}(\Gamma) = L(\Gamma)< (1+\epsilon){n-d+t \choose t}$, to complete the proof of the two claims, it suffices to show that $\K$ is $j$-neighborly for $j \leq t$, or that $r=t+1$.  Observe that, since $\K$ is $(r-1)$-neighborly but not $r$-neighborly, $\frac{f_{r-1}}{f_{r-2}} < \frac{{n \choose r}}{{n \choose r-1}}$, or $q_{r-1} < n-r+1$.  If $r < t+1$, it follows that $$L(\K) = \frac{n(n-1)\ldots(n-d+1)}{(n-q_0+1)\ldots(n-q_{d-1}+1)} < \frac{r}{r+1}(1+\epsilon){n-d+t \choose t}.$$  By choosing $\epsilon$ sufficiently small and using $f_{d-1} > (1-\epsilon){n \choose t}$, we derive a contradiction to the assumption that $\K$ attains the multiplicity lower bound.  Hence $\K$ is $(t-1)$-neighborly.
\endproof

Our first corollary looks at a class of simplicial complexes with many vertices.  The following corollary holds since a complex satisfying the conditions will violate Condition 2 of Theorem \ref{LLB}.

\begin{corollary}
\label{LargeH}
Let $\K$ be a $(d-1)$-dimensional simplicial complex with $n$ vertices satisfying $h_i(\K) \leq h_{d-i}(\K)$ for $i \leq d/2$.  Suppose that if $d$ is even, $\K$ is not $(d/2)$-neighborly.  Then for $n$ sufficiently large, $\K$ does not attain the multiplicity lower bound.
\end{corollary}

Corollary \ref{LargeH} applies to ear-decomposable complexes (as defined in \cite{Chari}) with sufficiently many vertices, provided that if $d$ is even, then $\K$ is not $(d/2)$-neighborly.  It is conjectured (Problem 4.2 of \cite{Swartz05}) that all $2$-Cohen-Macaulay complexes satisfy $h_i \leq h_{d-i}$ for $i \leq d/2$.

Our next two results consider equality for the multiplicity lower bound conjecture for Gorenstein complexes with many vertices.

\begin{corollary}
\label{GStar}
Let $\K$ be a $(d-1)$-dimensional Gorenstein complex with $n$ vertices.  For sufficiently large $n$, $\K$ attains the multiplicity lower bound if and only if $d$ is even and $\K$ is $(d/2)$-neighborly.
\end{corollary}

\begin{proof}
Since a Gorenstein complex is the simplicial join of a Gorenstein* complex and a simplex, we may assume without loss of generality that $\K$ is Gorenstein*.  Then $\K$ satisfies the Dehn-Sommerville equations: $h_i = h_{d-i}$.  The proof is completed by Corollary \ref{LargeH} and the following lemma.
\end{proof}

\begin{lemma}
\label{d2Neighborly}
Let $\K$ be a $(d-1)$-dimensional Gorenstein* complex with $d$ even and $n$ vertices ($n$ not necessarily large), and suppose $\K$ is $(d/2)$-neighborly.  Then $\K$ has a pure resolution.  Since $\K$ is Cohen-Macaulay, $\K$ attains both bounds of the multiplicity conjecture.
\end{lemma}
\proof First we describe the minimal and maximal shift sequences of $\K$.  If $W \subset V(\K)$ and $|W| \leq d/2$, then $\K[W]$ is a simplex and so $\tilde{H}_k(\K[W])=0$ for all $k$.  This implies $m_1 \geq d/2+1$ and $M_1 \geq d/2+1$.  Since $\K$ is Gorenstein*, the minimal free resolution of $S/I_\K$ is self-dual and we can determine the complete minimal and maximal shift sequences: $m = M = (d/2+1, d/2+2, \ldots, n-d/2-1, n)$. \endproof

Corollary \ref{GStar} applies to the class of simplicial spheres.  The condition that $n$ is sufficiently large is necessary in the second statement of Corollary \ref{GStar}.  For $d > 2$, the cross polytope is not $(d/2)$-neighborly, but it has a pure resolution.

We explore the case $r=2$ Theorem \ref{LLB} in greater detail in Section \ref{PureResolutions}.

\section{Pure Resolutions}
\label{PureResolutions}
In this section, our goal is to classify the set of simplicial complexes that have a pure resolution.  We give a complete characterization of flag simplicial complexes with a pure resolution and mention the non-flag case.
\begin{theorem}
\label{PureFlag}
Let $\K$ be a flag complex with a pure resolution.  Then one of the following applies. \newline
1) $\K$ is $1$-Leray. \newline
2) $\K$ is the join of a cycle and a simplex. \newline
3) $\K$ is the join of a cross polytope and a simplex.
\end{theorem}

To prove Theorem \ref{PureFlag}, we first need this technical lemma.

\begin{lemma}
\label{PureFlagLemma}
Let $\K$ be a $(d-1)$-dimensional flag complex with $n$ vertices and a pure resolution, and suppose $\K$ is not $1$-Leray.  Then there exists $W \subset V(\K)$ such that $|W| = Q_1(\K)+1$ and $\tilde{H}_1(\K[W]) \neq 0$.
\end{lemma}

\proof Since $\K$ is not $1$-Leray, there exists $r$ with $M_r > r+1$.  Suppose $r$ is chosen minimally.  Note that if $r \leq n-d$, then $r = Q_1(\K)-1$.  Then there exist $W \subset V(\K)$ and $p \geq 1$ such that $|W| = r+1+p$ and $\tilde{H}_p(\K[W]) \neq 0$.  Observe that if there exists another value $p' > p$ with this property, then $m_{r} \leq p$ and $M_{r} \geq p'$, contradicting the hypothesis that $\K$ has a pure resolution.  Hence there is only one such $p$.  Also observe by minimality of $r$ that if $T \subset V(\K)$ and $|T| < r+1+i$, then $\tilde{H}_i(\K[T]) = 0$ if $i > 0$.  We will prove that $p=1$.  Then suppose, by way of contradiction, $p > 1$.

Choose $v \in W$.  Then $\K[W] = \st_{\K[W]}(v) \cup \K[W-v]$, and $\st_{\K[W]}(v) \cap \K[W-v] = \lk_{\K[W]}(v)$.  Consider the following Mayer-Vietoris sequence:
$$
\ldots \rightarrow \tilde{H}_{p}(\st_{\K[W]}(v)) \oplus \tilde{H}_{p}(\K[W-v]) \rightarrow \tilde{H}_p(\K[W]) \rightarrow \tilde{H}_{p-1}(\lk_{\K[W]}(v)) \rightarrow \ldots.
$$
Since $\st_{\K[W]}(v)$ is a cone and $|W-v| < r+1+p$, $\tilde{H}_{p}(\st_{\K[W]}(v)) \oplus \tilde{H}_{p}(\K[W-v]) = 0$.  We conclude $\tilde{H}_{p-1}(\lk_{\K[W]}(v)) \neq 0$.

Let $T = V(\lk_{\K[W]}(v))$.  Since $|T| \leq r+p$, $\tilde{H}_{p-1}(\K[T]) = 0$.  Consider $\K[T \cup v] = \st_{\K[W]}(v) \cup \K[T]$ and $\st_{\K[W]}(v) \cap \K[T] = \lk_{\K[W]}(v)$.  It follows from the Mayer-Vietoris sequence
$$
\ldots \rightarrow \tilde{H}_p(\K[T \cup v]) \rightarrow \tilde{H}_{p-1}(\lk_{\K[W]}(v)) \rightarrow \tilde{H}_{p-1}(\st_{\K[W]}(v)) \oplus \tilde{H}_{p-1}(\K[T]) \rightarrow \ldots
$$
that $\tilde{H}_p(\K[T \cup v]) \neq 0$.

It follows that $|T \cup v| = |W|$, or that if $u \in W$, $vu$ is an edge in $\K$.  Repeating this argument for each $v \in W$, the graph of $\K[W]$ must be the complete graph, and hence $\K[W]$ is a simplex.  This is a contradiction, so we conclude $p=1$ and $\tilde{H}_1(\K[W]) \neq 0$. \endproof

\proofof{Theorem \ref{PureFlag}} Suppose that $\K$ is not $1$-Leray, so that by Lemma \ref{PureFlagLemma} there exists $W \subset V(\K)$ such that $|W| = r$ and $\tilde{H}_1(\K[W]) \neq 0$.  Also suppose $r$ is chosen minimally.  Since $\K$ is flag, $r \geq 4$.  We will consider two cases: first $r = 4$ and then $r > 4$.

If $r=4$, $\tilde{H}_0(\K[T]) = 0$ if $|T| = 3$ by the condition that $\K$ has a pure resolution.  If $x_ix_j \in I_\K$, then $x_ix_k \not\in I_\K$ if $k \neq j$.  Then, after relabeling the vertices, $I_\K = (x_1x_2,x_3x_4,\ldots x_{2r-1}x_{2s})$ for some $s$.  In that case, $\K$ is the join of the boundary of the $s$-dimensional cross polytope and a simplex.

Now suppose $r \geq 5$.  Let $W = \{u_1,u_2,\ldots,u_r\}$ so that the edges in $\K[W]$ are $(u_1u_2,\ldots,u_ru_1)$.  If $\K$ has $r$ vertices, then $\K = \K[W]$ is a cycle.  Otherwise, let $v \in V(\K) - W$.  If $vu_i$ is not an edge in $\K$, then $\K[W \cup \{v\} - \{u_{i-1},u_{i+1}\}]$ is disconnected (define $u_0 := u_r$ and $u_{r+1} := u_1$), and $|W \cup \{v\} - \{u_{i-1},u_{i+1}\}| = r-1$, a contradiction to the hypothesis that $S/I_\K$ has a pure resolution.  Hence $vu_i$ is an edge for $1 \leq i \leq r$.  If $v' \neq v \in V(\K)-W$, then $vv'$ is an edge in $\K$: otherwise, $\K[u_1,v,u_3,v']$ consists of only the edges $u_1v,vu_3,u_3v',v'u_1$ and $\tilde{H}_1(\K[u_1,v,u_3,v']) \neq 0$, contradicting the assumption $r \geq 5$.  Hence for vertices $y,z \in V(\K)$, $yz$ is not an edge if and only if $y,z \in W$ and $yz$ is not an edge in $\K[W]$.  Since $\K$ is a flag complex, we conclude that $\K$ is the join of $\K[W]$ and a simplex, or that $\K$ is the join of a cycle and a simplex. \endproof

If $\K$ is not a flag complex and the $1$-skeleton of $\K$ is not the complete graph, then $m_1(\K) = 2$ and $M_1(\K) > 2$, so $\K$ does not have a pure resolution.  Thus, Theorem \ref{PureFlag} classifies all complexes that have a pure resolution and do not have a complete $1$-skeleton.

If $I \subset W$ is a quadratic monomial ideal, then $I$ has a pure resolution if and only if the polarization of $I$, labeled $I'$, is the Stanley-Reisner ideal of one of the complexes described in Theorem \ref{PureFlag}.  $I'$ satisfies Condition 3 if and only if for all pairs of generators $(m_1,m_2)$ of $I$, $LCM(m_1,m_2)=1$; that is, $S/I$ is a \textit{complete intersection}.

Suppose $x_i^2 \in I$ and $x_{i,1}x_{i,2}$ is the polarization of $x_i^2$.  Then $x_{i,2}y \not\in I'$ if $y \neq x_{i,1}$.  It follows that $I'$ cannot be the Stanley-Reisner ideal of the join of a cycle $C$ and a simplex, with $C$ consisting of more than $5$ vertices, if $I$ is not a squarefree monomial ideal.

Suppose $I'$ is the Stanley-Reisner ideal of a complex $\K$ satisfying Condition 1 of Theorem \ref{PureFlag}.  In this case $S/I$ is also a $1$-regular ring, and we call the minimal free resolution of $S/I$ a \textit{linear resolution}.  If $S/I$ has a linear resolution, then $S'/I'$ and hence $S/I$ satisfy the multiplicity upper bound conjecture by the well-known fact that if $\K$ is a $(d-1)$-dimensional $1$-Leray simplicial complex with $n$ vertices, then $f_{d-1}(\K) \leq n-d+1$.  (See \cite{Kalai84} for a generalization of this fact.)  Furthermore, $S/I$ attains the upper bound if and only if $S/I$ is Cohen-Macaulay.

The three categories in Theorem \ref{PureFlag} have natural generalizations beyond flag complexes.
\begin{proposition}
If $\K$ satisfies one of the following conditions, then $\K$ has a pure resolution. \newline
1) $S/I_\K$ has an $r$-linear resolution.  That is, $\K$ is $(r-1)$-neighborly and is $(r-1)$-Leray. \newline
2) $\K$ is the join of $r$-neighborly $(2r-1)$-dimensional Gorenstein* complex and a simplex. \newline
3) $\K$ is the join of copies of the boundary of an $r$-simplex and a simplex.
\end{proposition}

However, these conditions are not necessary.  Consider the $2$-dimensional complex $\K$ with vertex set $[7]$ and facets $$\{124, 126, 127, 135, 136, 137, 145, 147, 156, 234,$$ $$235, 236, 245, 257, 267, 346, 347, 357, 456, 467, 567\}.$$  One can verify that $m(\K) = M(\K) = (3,4,6,7)$, $\K$ is Cohen-Macaulay, and $f_{2}(\K) = 12$, but $\K$ does not satisfy any of the conditions above.

\section*{Acknowledgments}
The author wishes to thank Isabella Novik for all her help in preparation of this paper.  Also, the author wishes to thank Uwe Nagel and Ezra Miller for help with Lemma \ref{IncM}.

\end{document}